\documentclass[]{article}

\usepackage[left=3cm, right=3.5cm, top=3cm]{geometry}
\usepackage{amssymb}
\usepackage{amsmath}
\usepackage{amsfonts}
\usepackage{amsthm}
\usepackage{mathtools}
\usepackage{listings}
\usepackage{latexsym}
\usepackage{booktabs}
\usepackage{multirow}
\usepackage{graphicx}
\usepackage{xcolor}
\usepackage{epsfig}
\usepackage{epstopdf}
\usepackage{url}
\usepackage{relsize}
\usepackage{hyperref}
\usepackage{textcomp}
\usepackage{indentfirst}
\usepackage{tipa}
\usepackage{algorithm}
\usepackage{algpseudocode}
\usepackage{subcaption}
\usepackage{dsfont}
\usepackage{float}
\usepackage[mathscr]{euscript}
\usepackage{enumitem}
\usepackage[title]{appendix}
\usepackage{bm}
\usepackage[backend=bibtex,style=numeric,citestyle=numeric,giveninits=true,sortcites=true,maxbibnames=4]{biblatex}

\renewbibmacro{in:}{}
\usepackage{color}
\definecolor{deepblue}{rgb}{0,0,0.5}
\definecolor{deepred}{rgb}{0.6,0,0}
\definecolor{deepgreen}{rgb}{0,0.5,0}

	\definecolor{DarkBlue}{rgb}{0.00,0.00,0.55}
	\definecolor{Black}{rgb}{0.00,0.00,0.00}
	\hypersetup{
		linkcolor = DarkBlue,
		anchorcolor = DarkBlue,
		citecolor = DarkBlue,
		filecolor = DarkBlue,
		colorlinks  = true,
		urlcolor = Black
	}
	
\graphicspath{ {./Figures/},{./} }
\DeclareGraphicsExtensions{.eps,.pdf,.png,.jpg}
	
\newtheorem{theorem}{Theorem}[section]

\newtheorem{corollary}[theorem]{Corollary}
\theoremstyle{definition}

\newtheorem{example}{Example}[section]
\theoremstyle{remark}
\newtheorem{remark}{Remark}[section]

\newtheorem{assumption}{Assumption}[section]
		
\newcommand{\TheTitle}{Multi-output multilevel best linear unbiased estimators via
	semidefinite programming}

\newcommand{\TheAuthors}{M.~Croci, K.~E.~Willcox, S.~J.~Wright}

\title{{\TheTitle}\thanks{\textbf{Funding:} 
    MC thanks Nicole Aretz for her helpful feedback on Section \ref{sec:numerical_results}.
    MC and KW's research is based upon work
    supported by the Department of Energy NNSA under Award Number
    DE-NA0003969. SW acknowledges support from a J.~T.~Oden Faculty
    Fellowship from the Oden Institute at the University of Texas at
    Austin as well as NSF Awards DMS 2023239 and CCF 2224213, DOE via
    subcontract 8F-30039 from Argonne National Laboratory, and AFOSR
    via subcontract UTA20-001224 from the University of Texas at
    Austin. KW and SW additionally acknowledge support from the
    AFOSR MURI on Machine Learning of Physics-based Systems, AFOSR Grant FA9550-21-1-0084.}}

\author{
  M. Croci\thanks{Oden Institute for Computational Engineering and Sciences, University of Texas at Austin, Austin, TX, USA. (\textbf{\url{matteo.croci@austin.utexas.edu}}), (\textbf{\url{kwillcox@oden.utexas.edu}}).}
  \and
  K.~E.~Willcox\footnotemark[2]
  \and
  S.~J.~Wright\thanks{Computer Science Department, University of Wisconsin-Madison, Madison, WI, USA. (\textbf{\url{swright@cs.wisc.edu}}).}
}

\usepackage{amsopn}
\DeclareMathOperator{\diag}{diag}
\DeclareMathOperator{\E}{\mathbb{E}}
\DeclareMathOperator{\V}{\mathbb{V}}

\DeclareMathOperator{\eps}{{\varepsilon}}
\DeclareMathOperator{\beps}{{\bm{\varepsilon}}}

\renewcommand{\kappa}{\chi}

\usepackage{todonotes}
\definecolor{myblue}{RGB}{135, 206, 250}
\definecolor{mygreen}{RGB}{144,238,144}
\definecolor{myyellow}{RGB}{255, 219, 88}

\newcommand{\R}{\mathbb{R}}
\newcommand{\N}{\mathbb{N}}

\ifpdf
\hypersetup{
  pdftitle={\TheTitle},
  pdfauthor={\TheAuthors}
}
\fi

\begin{document}

\maketitle

\begin{abstract}
    Multifidelity forward uncertainty quantification (UQ) problems
    often involve multiple quantities of interest and heterogeneous
    models (e.g., different grids, equations, dimensions, physics, surrogate
    and reduced-order models). While computational efficiency is
    key in this context, multi-output strategies in
    multilevel/multifidelity methods are either sub-optimal or
    non-existent.
    In this paper we extend multilevel best linear unbiased estimators (MLBLUE)
    to multi-output forward UQ problems and we
    present new semidefinite programming formulations for their optimal setup. Not only do these formulations yield the optimal number of samples required, but also the optimal selection of low-fidelity models to use.
    While existing MLBLUE approaches are single-output only and require a non-trivial nonlinear
    optimization procedure, the new multi-output formulations can be solved reliably and efficiently.
    We demonstrate the efficacy of
    the new methods and formulations in practical UQ problems with
    model heterogeneity.
\end{abstract}


%\paragraph{AMS subject classification:}
%	FIXME.

\section{Introduction}
\label{sec:intro}

A central task in forward uncertainty quantification (UQ) is estimating the expectation of output quantities of interest (QoI). When the same QoI is the output of different models of varying fidelity, multilevel and multifidelity Monte Carlo algorithms \cite{heinrich2001multilevel,giles2008multilevel,NgWillcox2014multifidelity,giles2015multilevel,haji2016multi,peherstorfer2016optimal,gorodetsky2020generalized,schaden2020multilevel,schaden2021asymptotic} efficiently estimate its expectation via sampling: by exploiting the correlations between models, these algorithms reduce the estimator variance for a fixed cost (or, equivalently, reduce the total cost to obtain an estimator with a given level of variance). 
Multilevel and multifidelity methods require selecting which models to sample and the number of samples to be drawn, a choice that can impact efficiency if sub-optimal. But setting up these methods optimally is a non-trivial operation, especially when UQ problems involve multiple QoIs and heterogeneous models, as is often the case in many engineering and scientific applications. We say models are ``heterogeneous'' if they are constructed through different approaches (e.g., varying grids, equations, dimensions, or physics, or using surrogate and reduced-order models) or if their output structure varies (e.g., if models have different QoIs). Since it is crucial to design estimators that are as efficient as possible, in this paper we extend multilevel best linear unbiased estimators (MLBLUEs) \cite{schaden2020multilevel,schaden2021asymptotic} to the multi-output case and propose a fast, reliable strategy for their optimal model selection and sample allocation.

There are many multilevel and multifidelity estimators available in the literature \cite{heinrich2001multilevel,giles2008multilevel,NgWillcox2014multifidelity,giles2015multilevel,haji2016multi,peherstorfer2016optimal,gorodetsky2020generalized,schaden2020multilevel,schaden2021asymptotic}, each differing in the way they exploit the statistical relations between models.  A key ingredient of any multilevel/multifidelity algorithm is a strategy for optimal model
selection and sample allocation: given a set of available models with
different characteristics, it is important to determine which models
actually lead to a performance benefit, and how much these
models should be sampled. The optimal \emph{model selection and sample
	allocation problem} (MOSAP) depends on both the method and on the
set of available models. In some specific cases the MOSAP is trivial
to solve; the model selection largely simplifies when all models
belong to the same discretization hierarchy
\cite{giles2008multilevel,giles2015multilevel}.
However, in the
general case in which models are highly heterogeneous, the MOSAP is an
NP-hard nonlinear mixed-integer programming problem, and its
complexity increases exponentially with the number of available models.

In this paper we focus on MLBLUEs, and specifically on techniques for
solving their MOSAP.
MLBLUEs were presented by Schaden and Ullmann in
\cite{schaden2020multilevel,schaden2021asymptotic}, and have the
appealing property of being provably optimal among all multilevel
linear unbiased estimators\footnote{This includes most multilevel and multifidelity methods and all previously mentioned methods.}
\cite{schaden2020multilevel,schaden2021asymptotic}. Furthermore, while multilevel Monte Carlo (MLMC) \cite{heinrich2001multilevel,giles2008multilevel,giles2015multilevel} can
only couple\footnote{We denote a group of models to be ``coupled'' if
  the algorithm samples all these models with the same input.} models in
pairs, and multifidelity Monte Carlo (MFMC) \cite{NgWillcox2014multifidelity,peherstorfer2016optimal} requires all models to be
coupled together, MLBLUEs have no coupling restrictions and can work
with any combination of models \cite{schaden2020multilevel}.

Nevertheless, setting up a MLBLUE presents challenges.
The MOSAP formulation in \cite{schaden2020multilevel} is a (weakly) convex nonlinear
problem with linear constraints, but its objective function is defined
via the inverse of a matrix that may be singular. This is handled by
adding a small positive term to the diagonal of the matrix, but the
resulting problem is still not strongly convex in general and may be
ill-conditioned, making the MLBLUE MOSAP a hard problem to solve.
For instance, the generic optimization solver \texttt{fmincon} is
used in the numerical tests of \cite[Section~6]{schaden2020multilevel},
but it fails in several cases. Another limitation of the existing MLBLUEs
is that they have been
designed only for a single output. Extending these estimators to multiple
QoIs, while trivial in theory, can be very challenging in practice
since any difficulties in the MLBLUE MOSAP optimization are likely to
be exacerbated in the multi-output case.

In this paper we make the following contributions.
\begin{itemize}
\item We reformulate the MLBLUE MOSAP as a semidefinite programming
  (SDP) problem. SDPs and linear programming are both conic
  optimization problems, a special class of convex optimization
  \cite{vandenberghe1996semidefinite}. SDPs have been an area of
  intense study in the optimization community since the mid-1990s, and
  powerful algorithms and software have been developed
  \cite{vandenberghe1996semidefinite,dsdp5,tutuncu2003solving,andersen2000mosek,diamond2016cvxpy}.
  Our SDP reformulation does not require
  taking the inverse of a possibly ill-conditioned matrix,
   nor does it require a positive shift to
  ensure positive definiteness of this matrix. We prove the
  equivalence between the SDP reformulation and the original MOSAP
  problem with zero shift from \cite{schaden2020multilevel}.
	
\item The original MLBLUE MOSAP minimizes the estimator variance for
  fixed cost \cite{schaden2020multilevel}. We present a SDP
  reformulation minimizing the MLBLUE cost for a given statistical
  error tolerance. Depending on the UQ problem, either formulation may
  be useful, and they can both be solved efficiently.

\item We introduce a multi-objective MOSAP SDP formulation that permits computation of 
  Pareto points defined by the two objectives of statistical error and
  cost. This is the first such formulation in the literature and is of practical relevance in cases where an appropriate limit on the computational budget is not known \emph{a priori}.
	
\item We extend MLBLUEs to the multi-output case and present an SDP
  reformulation for the multi-output MOSAP that can be solved
  efficiently. We show that this multi-output estimator outperforms
  MLMC and MFMC in numerical experiments.
  
\item We demonstrate numerically that MLBLUEs maintain their efficiency when the prohibitive cost of high-fidelity models imposes sample restrictions and/or prevents the estimation of all model correlations.
\end{itemize}

This paper is structured as follows. In Section~\ref{sec:background}
we introduce MLBLUEs and their MOSAP and we discuss their
properties. 
We reformulate the single-output MLBLUE MOSAP as an SDP in Section~\ref{sec:SDP_reformulation}, where we also present a new formulation in which the total MLBLUE cost is minimized subject to a statistical error constraint and a new multi-objective MOSAP formulation.
In Section~\ref{sec:multi-output} we extend MLBLUEs and their MOSAPs to the
multi-output case. 
Numerical experiments are presented in Section~\ref{sec:numerical_results}, and Section~\ref{sec:conclusions} contains some concluding remarks.

\paragraph{Notation.}
We adopt the following notation:

\textbf{Scalars, vectors, matrices, sets.} We indicate scalars with lower-case letters (e.g., $a$, $b$, $c$), vectors with lower-case bold letters (e.g., $\bm{a}$, $\bm{b}$, $\bm{c}$), matrices with upper-case letters (e.g., $A$, $B$, $C$), and sets with calligraphic upper-case letters (e.g., $\mathcal{A}$, $\mathcal{B}$, $\mathcal{C}$). Random variables and vectors are indicated in the same way. We indicate with $2^{\mathcal{A}}$ the power set of $\mathcal{A}$, i.e., the set of all possible subsets of $\mathcal{A}$.

\textbf{Linear algebra.} Given a matrix $A$, we indicate with $\text{null}(A)$, $\text{col}(A)$, and $\text{row}(A)$ its nullspace, column and row space respectively. Furthermore, we denote with $A^T$, $A^{-1}$ and $A^\dagger$ its transpose, inverse, and Moore-Penrose pseudo-inverse, and we write $A\succeq 0$ to indicate that $A$ is positive semi-definite. Let $\{a_i\}_{i=1}^\ell$, $\{\bm{a}_i\}_{i=1}^\ell$, and $\{A_i\}_{i=1}^\ell$ be ordered sets of scalars, vectors and matrices respectively. We indicate with $[a_i]_{i=1}^\ell$, $[\bm{a}_i]_{i=1}^\ell$, and $[A_i]_{i=1}^\ell$ the vectors and the matrix obtained by vertically stacking the elements (assuming that matrix dimensions are compatible).

\textbf{Random input parameters.} We indicate with $\omega$ a generic input parameter which is modeled as random from a known distribution. We indicate independent identically distributed input parameter samples with different subscripts and superscripts, e.g., $\omega_i$, $\omega_j$ or $\omega_i^k$, $\omega_j^s$, where the samples are independent if $i\neq j$ or if $k \neq s$.

\textbf{Random variables, expectation, variance, covariance.} We indicate random variables and vectors as functions of the random input $\omega$, e.g. $p(\omega)$ for a scalar random variable or $\bm{p}(\omega)$ for a random vector. We indicate with $\E[\cdot]$, $\V[\cdot]$ and $C(\cdot,\cdot)$ the expectation, variance, and covariance operators. We always assume that the random quantities involved have finite variance.

\section{Background on multilevel best linear unbiased estimators}
\label{sec:background}

To set the scene for the later discussion, we introduce
the single-output MLBLUEs by Schaden and Ullmann \cite{schaden2020multilevel,schaden2021asymptotic},
and discuss their formulation of the MLBLUE MOSAP. This section summarizes the technical approach presented in \cite{schaden2020multilevel}.

In forward UQ the objective is to approximate the expectation $\E[p(\omega)]$ of a scalar QoI $p(\omega)$ which is modeled as a random variable of finite variance. For this purpose, we consider a set of $\ell$ different models $\mathcal{M} = \{\mathcal{M}_i\}_{i=1}^{\ell}$ of different accuracy/fidelities, so that for all $i=1,\dots,\ell$, model $\mathcal{M}_i$ admits the random variable $p_i(\omega)$ as output, where $p_i\approx p$. Here we indicate with $\mathcal{M}_1$ the highest-fidelity model whose output $p_1$ best approximates $p$. Model $\mathcal{M}_1$ typically has the highest cost across the model set.

\begin{remark}
	We present MLBLUEs in a \emph{multifidelity setting} \cite{peherstorfer2018survey} in which $p_1$ is taken as the best model we have access to. In this setting it is not possible to refine $p_1$ further. Since the bias error $|\E[p-p_1]|$ cannot be reduced, it is only taken into account when determining a suitable statistical error tolerance and otherwise ignored. This is in contrast to what we call a \emph{multilevel setting} \cite{giles2008multilevel} in which it is always possible to refine models further to reduce the bias. We remark that MLBLUEs -- as well as MLMC and MFMC -- can be used in both settings \cite{schaden2020multilevel,schaden2021asymptotic}.
\end{remark}

Given a set $\mathcal{G} \subseteq 2^{\mathcal{M}} \setminus \emptyset$ of allowed model groupings, MLBLUEs approximate $\E[p_1]$ by combining independent samples of the model groups in $\mathcal{G}$, where by ``sample of a model group'' we mean that the models in the group are sampled with the same input (i.e., same $\omega$). We order model groups by size and then lexicographically by model, and we indicate the $k$-th group in $\mathcal{G}$ with $\mathcal{G}_k$ for $k=1,\dots,|\mathcal{G}|$. For instance, for $\ell=3$ and all groupings allowed, we would have $\mathcal{G}= \{\mathcal{G}_k\}_{k=1}^{7}$, where
\begin{align*}
\mathcal{G}_1&=\{\mathcal{M}_1\},\quad \mathcal{G}_2 =\{\mathcal{M}_2\},\quad \mathcal{G}_3=\{\mathcal{M}_3\},
\\
\mathcal{G}_4&=\{\mathcal{M}_1,\mathcal{M}_2\},\quad \mathcal{G}_5=\{\mathcal{M}_1,\mathcal{M}_3\},\quad \mathcal{G}_6=\{\mathcal{M}_2,\mathcal{M}_3\},\\ \mathcal{G}_7&=\{\mathcal{M}_1,\mathcal{M}_2,\mathcal{M}_3\}.
\end{align*}

Now, let $\bm{\mu} = [\E[p_i]]_{i=1}^{\ell}\in\R^{\ell}$ be the vector of the expectations of all model outputs. Then a sample $\bm{p}^k(\omega)=[p_i(\omega)]_{M_i \in \mathcal{G}_k}$ from the $k$-th group $\mathcal{G}_k\in\mathcal{G}$ satisfies the linear model
\begin{align}
\label{eq:linear_model}
\bm{p}^k(\omega) = \E[\bm{p}^k] + (\bm{p}^k(\omega) - \E[\bm{p}^k]) = R_k\bm{\mu} + \beps^k(\omega),
\end{align}
where $R_k\in\R^{|\mathcal{G}_k|\times \ell}$ is a boolean restriction matrix (made up of rows of the $\ell \times \ell$ identity matrix corresponding to the elements of $\mathcal{G}_k$) such that $R_k\bm{\mu} = \E[\bm{p}^k]$, and $\beps^k=\bm{p}^k(\omega) - \E[\bm{p}^k]\in\R^{|\mathcal{G}_k|}$ is a zero-mean random vector with covariance matrix $C_k = C(\bm{p}^k,\bm{p}^k)$.

The starting point for constructing a MLBLUE is to build on \eqref{eq:linear_model} and account for multiple independent samples from all feasible groups.
We define the set of group indices $\mathcal{K}:=\{1,\dots,|\mathcal{G}|\}$, the number of samples from each group $n_k$, $k \in \mathcal{K}$, and the sample set for each group to be
$\mathcal{J}^k=\{1,\dots,n_k\}$.
The vector of these sample sizes --- the key variable to be determined in our MOSAP formulation --- is $\bm{n} := \left[  n_k \right]_{k \in \mathcal{K}} \in\N^{|\mathcal{G}|}$.
% Let $\bm{n}\in\N^{|\mathcal{G}|}$ be a vector such that $n_k$ is the number of independent samples of model group $k$, and define the sets $\mathcal{K}=\{1,\dots,|\mathcal{G}|\}$ and $\mathcal{J}^k=\{1,\dots,n_k\}\ \forall k\in \mathcal{K}$.
Using this notation, we can extend \eqref{eq:linear_model} by vertically stacking all samples from all groups into a vector $\bm{p}$ such that $\bm{p}=[\bm{p}^k(\omega_j^k)]_{j\in \mathcal{J}^k, k\in \mathcal{K}}$, where $\bm{p}^k(\omega_j^k)$ denotes the $j$-th sample of model group $k$ and $\bm{p}^k(\omega_j^k)$ and $\bm{p}^s(\omega_i^s)$ are independent if $s\neq k$ or if $i\neq j$. The vector $\bm{p}$ satisfies the linear model
\begin{subequations}
  \label{eq:linear_model_full}
\begin{align}
\bm{p}(\omega) &= R\bm{\mu} + \beps(\omega),\\ \text{where}\quad R &= [R_k]_{j\in \mathcal{J}^k, k\in \mathcal{K}},\quad \beps=[\beps^k(\omega_j^k)]_{j\in \mathcal{J}^k, k\in \mathcal{K}}.
\end{align}
\end{subequations}
By construction, $\beps$ is a zero-mean vector with block-diagonal covariance since it contains $n_k$ i.i.d.~samples of $\beps^k$ for all $k\in \mathcal{K}$ ($\mathcal{J}^k$ is the set of sample indices and $|\mathcal{J}^k|=n_k$). Accordingly, the matrix $R$ contains $n_k$ copies of $R_k$ for each $k \in \mathcal{K}$. Similarly, we  define
\begin{equation} \label{eq:Ceps}
  C_{\beps}:=C(\beps,\beps)=\diag(\{C_k\}_{j\in \mathcal{J}^k, k\in \mathcal{K}}),
\end{equation}
where this block-diagonal matrix contains $n_k$ copies of $C_k$, for each $k \in \mathcal{K}$.
So far $\bm{n}$ has been considered fixed. A MLBLUE is found by computing the best linear unbiased estimator (BLUE) of the linear model \eqref{eq:linear_model_full}, and then
optimizing it with respect to $\bm{n}$.

\begin{remark}
	Schaden \& Ullmann \cite{schaden2020multilevel} call the BLUE resulting for any linear model such as \eqref{eq:linear_model_full} that includes multilevel/multifidelity samples a MLBLUE. With the purpose of simplifying the exposition by reducing the number of acronyms, we here call MLBLUE only a \emph{model selection and sample allocation optimal} BLUE for $\bm{\mu}$ (termed sample allocation optimal MLBLUE or SAOB in \cite{schaden2020multilevel}).
\end{remark}

Note that there is a closed-form expression for the BLUE of \eqref{eq:linear_model_full} independently of whether $R$ and $C_{\beps}$ are full-rank or not, see e.g., \cite{gross2004general}. To proceed, we however make the following simplifying assumption that is also adopted in \cite{schaden2020multilevel}.
\begin{assumption}
	\label{ass:nonsingular_covariance}
	The covariance matrices $C_k$ are non-singular for all $k\in \mathcal{K}$.
\end{assumption}

\begin{remark}
Assumption \ref{ass:nonsingular_covariance} can be removed at the cost
of a more complex technical and algorithmic treatment, but is of
little importance in practice. Indeed, as shown in
\cite{gross2004general}, a singular $C_k$ implies that there
exists a linear combination of outputs from the different models in
the set $\mathcal{G}_k$ that is a.s.~constant. If this happens, then
at least one model is redundant.  Once the redundant models are
removed from $\mathcal{G}_k$, the resulting ensemble of models would
have a covariance matrix that is full-rank. Thus,
Assumption~\ref{ass:nonsingular_covariance} is not restrictive.
\end{remark}
\begin{remark}
In practice, the matrices $C_k$ often have to be estimated via pilot
samples.  Since a sample covariance computed with $n$ samples cannot
have more than rank $n$, one must therefore take enough samples to
ensure that the sample covariance is not artificially singular. If a
singular matrix is nevertheless obtained, one can either remove
the redundant models or to reset the zero eigenvalues to a small
positive constant.
\end{remark}

Under Assumption \ref{ass:nonsingular_covariance}, we can invoke the
generalized Gauss-Markov-Aitken theorem \cite{gross2004general} and
find the BLUE by solving the generalized least-square problem for the
estimation of the mean, which admits a closed-form solution given by
(cf.~\cite{rao2008linear})
\begin{align}
\label{eq:_MLBLUE_eqns}
\hat{\bm{\mu}} (\bm{n})= (R^TC_{\beps}^{-1}R)^{-1} R^TC_{\beps}^{-1}\bm{p}= \Psi(\bm{n})^{-1}\hat{\bm{y}}(\bm{n}),
\end{align}
where $\Psi(\bm{n})\in\R^{\ell\times \ell}$ and $\hat{\bm{y}}(\bm{n})\in
\R^\ell$ are given by
\begin{subequations}
  \label{eq:psi_y_def}
\begin{align}
  \Psi(\bm{n}) & = R^T C_{\beps}^{-1} R = \sum\limits_{k\in \mathcal{K}} n_k (R_k)^T (C_k)^{-1} R_k, \\
  \hat{\bm{y}}(\bm{n}) & = R^TC_{\beps}^{-1}\bm{p} = \sum\limits_{k\in \mathcal{K}} (R_k)^T (C_k)^{-1} \sum\limits_{j\in \mathcal{J}^k}\bm{p}^k(\omega_j^k).
\end{align}
\end{subequations}
Equation \eqref{eq:_MLBLUE_eqns} is the BLUE for the full vector $\bm{\mu}$. Since in practice we are only interested in estimating $\mu_1=\E[p_1]$, we can write the BLUE for ${\mu}_1$ and its variance as
\begin{align}
\label{eq:MLBLUE_HF}
\E[p_1]\approx \bm{e}_1^T\hat{\bm{\mu}} (\bm{n})=\bm{e}_1^T\Psi(\bm{n})^{-1}\hat{\bm{y}}(\bm{n}),\quad \V[\bm{e}_1^T\hat{\bm{\mu}}(\bm{n})] = \bm{e}_1^T\Psi(\bm{n})^{-1}\bm{e}_1,
\end{align}
where $\bm{e}_1=[1,0\dots,0]^T\in\R^{\ell}$. It can be shown \cite{gross2004general,rao2008linear} that this estimator for $\E[p_1]$ is also a BLUE and well-defined as long as $\bm{e}_1$ is in the column space of $R^T$, $\text{col}(R^T)$, which is always satisfied as long as we are sampling the high-fidelity model at least once.

To complete the construction of a MLBLUE, we must find the optimal
model selection and the optimal number of samples to draw for each
model group. Note that both problems can be addressed by finding the optimal value of $\bm{n}$ while allowing some of its entries to be zero: a model that is sampled zero times
is this way automatically discarded.
To obtain the MOSAP, we minimize the BLUE variance in
\eqref{eq:MLBLUE_HF} with respect to $\bm{n}$ for a fixed
positive computational budget $b$, i.e., we solve
\begin{align}
\label{eq:MLBLUE_MOSAP}
\min\limits_{\bm{n} \ge \bm{0}} \, \bm{e}_1^T\Psi^{-1}(\bm{n})\bm{e}_1,\quad \text{s.t.}\quad \bm{n}^T\bm{c}\leq b,
\end{align}
where $\bm{c}\in\R^{|\mathcal{G}|}$ is the cost vector such that $c_k$
is the cost of drawing one realization of model group $k$. Problem
\eqref{eq:MLBLUE_MOSAP} is typically solved by allowing $\bm{n}$ to
take real values, and then projecting the result to integers. After
continuous relaxation, problem \eqref{eq:MLBLUE_MOSAP} becomes a
weakly convex nonlinear optimization problem
\cite{schaden2020multilevel}. {In fact, it can be shown that even
  when $\Psi(\bm{n})$ is nonsingular for all $\bm{n}$, the Hessian of
  the objective function in \eqref{eq:MLBLUE_MOSAP} has rank at most
  $\ell$, typically much less than the number of model groups
  $|\mathcal{G}|$.} We remark that for the feasible set of
\eqref{eq:MLBLUE_MOSAP} to be non-empty the budget must satisfy
\begin{align}
{b\geq \min_{k :  \mathcal{M}_1\in \mathcal{G}_k} c_k,}
\end{align}
i.e., we have enough budget to sample the cheapest group containing $\mathcal{M}_1$ at least once. The resulting algorithm for constructing a MLBLUE is presented in Algorithm \ref{alg:MLBLUE}.

\begin{remark}
	We remark that problem \eqref{eq:MLBLUE_MOSAP} is only weakly convex and, as such, admits an infinite set of optimal solutions that are feasible and have the same objective value (cf.~\cite{schaden2020multilevel}). As a consequence, MLBLUEs are not unique.
\end{remark}

\begin{algorithm}[h!]
	\caption{Model and sample allocation optimal MLBLUE}
	\label{alg:MLBLUE}
	\begin{enumerate}[leftmargin=*,align=left]
		\item Given a set $\mathcal{G}$ of allowed model groupings, estimate the group correlations $\{C_k\}_{k\in\mathcal{K}}$, e.g., by taking $n_{\text{pilot}}\in\N$ pilot samples of all models.
		\item Solve the MLBLUE MOSAP \eqref{eq:MLBLUE_MOSAP} (or the improved MOSAP formulations we present in the next sections) to obtain the optimal sample vector $\bm{n}^*$.
		\item Sample the model groups according to $\bm{n}^*$ and stack the samples into a vector $\bm{p}^*$.
		\item Compute the MLBLUE as $\hat{\mu}_1^*=\bm{e}_1^T\Psi(\bm{n}^*)^{-1}R^TC_{\bm{\varepsilon}}^{-1}\bm{p}^*$, cf~equations \eqref{eq:psi_y_def} and \eqref{eq:MLBLUE_HF}.
	\end{enumerate}
\end{algorithm}

Problem \eqref{eq:MLBLUE_MOSAP} may be affected by severe
ill-conditioning. The reason is that $\Psi(\bm{n})$ is, in general,
only positive semi-definite and becomes singular (see next section)
whenever a model is discarded. Schaden and Ullmann address this
problem by adding a shift and replacing $\Psi(\bm{n})$ by
$\Psi_\delta(\bm{n}) = \Psi(\bm{n}) + \delta I$ for some $\delta>0$,
where the value of $\delta$ must be carefully chosen: a large $\delta$
artificially increases all entries of $\bm{n}$ and leads to
oversampling, while a small $\delta$ leads to
ill-conditioning. {(Although the matrices $\Psi_{\delta}(\bm{n})$
  are nonsingular for all $\bm{n}$, the rank of the Hessian of the
  objective function in \eqref{eq:MLBLUE_MOSAP} remains at most $\ell$,
  so the problem is still highly degenerate.)} We will show in the
next section how we can reformulate MLBLUEs so that no shift parameter
is required and so that no issues arise when $\Psi$ is singular or
when models are discarded.

\begin{remark}
	\label{rem:max_group_size}
	The number of possible model groupings grows exponentially fast with the number of models $\ell$, which might make the MOSAP solution challenging. For instance a model set of size $\ell>20$ would lead to $\bm{n}$ having over a million entries. A solution in this case is to restrict the set $\mathcal{G}$ of allowed groupings. For instance Schaden and Ullmann \cite{schaden2020multilevel} restrict the maximum size of the allowed groupings to a fixed integer $\kappa$ and show that in practice little accuracy or efficiency is lost by only using groups of size at most five.
\end{remark}

MLBLUEs have three advantages compared with other multilevel/multifidelity estimators. 1) MLBLUEs are provably optimal across all multilevel/multifidelity methods. 2) Unlike most multilevel/multifidelity methods, MLBLUEs are flexible in the sense that they can work with any selection of model groupings and with any model set structure. In contrast, methods such as MLMC and MFMC are restricted to specific model combinations (see Example \ref{ex:MLBLUE_flexible}). 3) The MLBLUE optimal model selection is automatically found together with the optimal sample allocation without need for a brute-force enumeration approach.

As we will see in Section \ref{sec:multi-output}, the multi-output MLBLUEs we introduce inherit the same advantages as the single-output case. We will also see how, thanks to the contributions of Sections \ref{sec:SDP_reformulation} and \ref{sec:multi-output}, we can add a fourth advantage: 4) The construction of multi-output MLBLUEs and the solution of their MOSAPs can be performed reliably and efficiently. Furthermore, we will numerically demonstrate in Section \ref{sec:numerical_results} how MLBLUE is robust to the grouping and sample restrictions arising in the common scenario in which the higher-fidelity models are prohibitive to sample.

\begin{example}
\label{ex:MLBLUE_flexible}
Let us clarify point 2) with an example. We take $\ell=3$ and, without loss of generality, we order the models by cost. Assuming that the best MLMC and MFMC estimators that can be constructed for this problem use all three models, we obtain the estimators
\begin{align*}
\hat{\mu}_{\text{MLMC}}&=\frac{1}{n_1}\sum_{i=1}^{n_1}\left(p_1(\omega_i^1)-p_2(\omega_i^1)\right) + \frac{1}{n_2}\sum_{i=1}^{n_2}\left(p_2(\omega_i^2)-p_3(\omega_i^2)\right) + \frac{1}{n_3}\sum_{i=1}^{n_3}p_3(\omega_i^3),\\
\hat{\mu}_{\text{MFMC}}&=\frac{1}{m_1}\sum_{i=1}^{m_1}p_1(\omega_i) +
\alpha_2\left( \frac{1}{m_2}\sum_{i=1}^{m_2}p_2(\omega_i) - \frac{1}{m_1}\sum_{i=1}^{m_1}p_2(\omega_i)\right) +
\alpha_3\left(\frac{1}{m_3}\sum_{i=1}^{m_3}p_3(\omega_i) - \frac{1}{m_2}\sum_{i=1}^{m_2}p_3(\omega_i)\right),
\end{align*}
where $0<n_1,n_2,n_3$, $0<m_1< m_2 < m_3$, and $\alpha_2$ and $\alpha_3$ are real parameters that depend on the covariances between the model outputs. We refer the reader to \cite{giles2015multilevel} and \cite{NgWillcox2014multifidelity} for further details on MLMC and MFMC respectively. Recalling that samples with different input parameters are independent, it is readily seen from the above estimators that MLMC and MFMC use the following model groupings:
\begin{align*}
\mathcal{G}_{\text{MLMC}} &= \left\{\{\mathcal{M}_1,\mathcal{M}_2\},\{\mathcal{M}_2,\mathcal{M}_3\},\{\mathcal{M}_3\}\right\},\\
\mathcal{G}_{\text{MFMC}} &= \left\{ \{\mathcal{M}_1,\mathcal{M}_2,\mathcal{M}_3\},\{\mathcal{M}_2,\mathcal{M}_3\},\{\mathcal{M}_3\} \right\},
\end{align*}
i.e., MLMC combines models in pairs, while MFMC samples all models together. We remark that the same holds for larger $\ell$: MLMC must work with pairs, and MFMC must be able to sample all models together. If this is not possible, e.g., because estimating correlations is too expensive or, in the multi-output case, simply because a model does not yield that QoI, then the model set must be pruned and adjusted so that an estimator can be constructed. On the other hand, MLBLUEs do not have such restrictions (nor most of the other model restrictions of these methods, cf.~\cite{giles2015multilevel,NgWillcox2014multifidelity}): it can work with any set of allowed groupings without needing to discard any models, and it will automatically find the best groups for the estimation problem at hand together with the optimal sample allocation. Generally speaking, while the model set structure could significantly affect the performance of other estimators (see e.g., \cite{gorodetsky2020generalized} for a discussion), MLBLUE always retains its optimality. This is why we say that MLBLUEs are flexible.
\end{example}

\section{Semidefinite programming reformulation of the MLBLUE model selection and sample allocation problem}
\label{sec:SDP_reformulation}

While the MLBLUE MOSAP is (weakly) convex, it is still non-trivial to solve in
practice, especially when $|\mathcal{G}|$ is large. Near-singularity of
$\Psi(\bm{n})$ causes the problem to be ill-conditioned, even when
the shift $\delta$ is added to form $\Psi_{\delta}(\bm{n})$.
Efficient methods should use second-order information, which can be
expensive to calculate. In this section, we show how the solution of the
MLBLUE MOSAP can be simplified and accelerated by reformulating
\eqref{eq:MLBLUE_MOSAP} as a semidefinite programming (SDP) problem.

Semidefinite programming is a powerful paradigm for formulating and
solving a wide range of problems in statistics, data science, and
control.  The problem class most relevant to the MOSAP is experimental
design (see for example \cite[Section~7.5]{BoyV03}) in which the
objective is some function of a matrix that is a weighted sum of
rank-1 matrices, with the weights being the variables in the
problem. (In the MOSAP, the matrices in the sum \eqref{eq:psi_y_def} that
defines $\Psi(\bm{n})$ are not rank-1, but the SDP formulation
techniques still apply.)  Efficient interior-point
algorithms have been available for SDP since the
mid-1990s, and are now included in commercial packages such as MOSEK
\cite{andersen2000mosek} and open-source packages such as \texttt{CVXOPT} \cite{cvxopt}.

The first step in the reformulation is to replace $\Psi^{-1}$ with the Moore-Penrose pseudo-inverse $\Psi^\dagger$ so that
\begin{align}
\hat{\bm{\mu}} = \Psi^\dagger\hat{\bm{y}},\quad C(\hat{\bm{\mu}},\hat{\bm{\mu}})=\Psi^\dagger.
\end{align}
This is a standard procedure whenever $\Psi$ is singular and still leads to a BLUE for $\bm{\mu}$ under Assumption \ref{ass:nonsingular_covariance} \cite{zyskind1969best,rao2008linear}, but has the advantage of not requiring the $\delta$ shift parameter used in \cite{schaden2020multilevel} (cf.~Section \ref{sec:background}) without leading to ill-conditioning or ill-posedness of the resulting MOSAP. However, using the pseudo-inverse on its own is still not sufficient to ensure that $\bm{e}_1^T\bm{\mu}$ is also BLUE. As previously said, we also need $\bm{e}_1$ to be in the column space of $R^T$ \cite{rao2008linear}. This property can be imposed by requiring that model $\mathcal{M}_1$ is sampled at least once. Problem \eqref{eq:MLBLUE_MOSAP} can then be equivalently reformulated as
\begin{align}
\label{eq:MLBLUE_MOSAP_1}
\min\limits_{\bm{n} \ge \bm{0}}\bm{e}_1^T\Psi^{\dagger}(\bm{n})\bm{e}_1,\quad \text{s.t.}\quad \bm{n}^T \bm{c}\leq b,\quad \bm{n}^T \bm{h} \geq 1,
\end{align}
where $\bm{h}\in\R^{|\mathcal{G}|}$ is a known boolean vector such that $h_k=1$ if and only if $\mathcal{M}_1\in \mathcal{G}_k$. The new constraint on the right-hand side imposes that the high-fidelity model is sampled at least once and thus ensures that the estimator is a well-posed BLUE. Since it is necessary for well-posedness of the resulting estimator, this constraint does not exclude any sensible solution of the MOSAP. In particular, the standard Monte Carlo estimator is still feasible for problem \eqref{eq:MLBLUE_MOSAP_1}, ensuring that the feasible set is non-empty (also see Theorem 3.6 in \cite{schaden2020multilevel}).

Before proceeding we need two auxiliary results. Under Assumption \ref{ass:nonsingular_covariance}, Schaden and Ullmann \cite{schaden2020multilevel} proved that if $\bm{p}$ contains samples from all $\ell$ models, then $\Psi$ is positive definite. We further extend their result to show that $\Psi$ can only be positive definite if all models are sampled, as well as providing a characterization of the null-space of $\Psi$.
\begin{theorem}
	\label{th:Psi_properties}
	Let Assumption \ref{ass:nonsingular_covariance} hold. Then $\Psi$ is positive definite if and only if $\bm{p}$ contains samples from all $\ell$ models. More specifically, let $\mathcal{V}_k=\{\bm{v}\in\R^\ell$ : $v_i=0$ $\forall i$ s.t. $\mathcal{M}_i\in \mathcal{G}_k\}$, then the null-space of $\Psi$ satisfies $\textnormal{null}(\Psi)=\bigcap_{n_k>0} \mathcal{V}_k$. In particular, if model $\mathcal{M}_i\in\mathcal{M}$ is never sampled, then the $i$-th row and column of $\Psi$ and $\Psi^\dagger$ are zero.
\end{theorem}
\begin{proof}
  Let $B_k=(R_k)^T (C_k)^{-1} R_k$ for all $k$. Since $C_k$ is positive definite for all $k$, then $B_k$ is positive semidefinite, and so is $\Psi$, cf.~\eqref{eq:psi_y_def}.
  Since $(C_k)^{-1}$ is positive definite and $R_k$ has full row rank,  $\bm{v} \in \R^\ell$ is in the null-space of $B_k$ if and only if $R_k\bm{v}=\bm{0}$, which only happens if $\bm{v}\in \mathcal{V}_k$. We then conclude that $\Psi$ is singular if and only if there exists $\bm{v}\in\R^\ell$ such that (recall: $\mathcal{K}=\{1,\dots,|\mathcal{G}|\}$)
	\begin{align}
	\label{eq:_thm1_aux}
	\bm{v}^T\Psi\bm{v} = \sum\limits_{k\in \mathcal{K}} n_k \bm{v}^TB_k\bm{v} = 0\quad\Leftrightarrow\quad \forall k\in \mathcal{K}\ \text{either}\ n_k = 0\ \text{or}\  \bm{v}\in \mathcal{V}_k,
	\end{align}
	since $n_k$ and $\bm{v}^TB_k\bm{v}$ are both non-negative for all $k$. In particular, this means that the null-space of $\Psi$ is given by $\bigcap_{n_k>0} \mathcal{V}_k$.
	
	We first prove that if a model is not sampled, then $\Psi$ is singular. Let $\bm{e}_i$ with $(\bm{e}_i)_j = \delta_{ij}$, ($\delta_{ij}$ is the Kronecker delta) be the canonical basis vector. If a model $\mathcal{M}_i$ is not sampled, we have that $n_k=0$ for all $k$ such that $\mathcal{M}_i\in \mathcal{G}_k$, and $\bm{e}_i\in \mathcal{V}_k$ for all $k$ such that $\mathcal{M}_i \notin \mathcal{G}_k$, therefore $\bm{e}_i$ is in the nullspace of $\Psi$, which is therefore singular with zeros in its $i$-th row and column. The same property holds for $\Psi^\dagger$, since $\Psi$ and $\Psi^\dagger$ have the same null-space.
	
	We now prove the reverse statement and we show that if all models are sampled, then $\Psi$ is invertible. In this case, for all $k$ corresponding to sampled groups we have $n_k>0$, and it is readily seen that since we are sampling all models $\bigcap_{n_k>0} \mathcal{V}_k = \{\bm{0}\}$. Hence the null-space of $\Psi$ is trivial and $\Psi$ is invertible.
\end{proof}

The second auxiliary ingredient is the following result by Albert \cite{albert1969conditions}:
\begin{theorem}[Result (i) in Theorem 1 in \cite{albert1969conditions}]
	\label{th:Schur_pseudoinverse}
	Let $A\in\R^{n\times n}$, $B\in \R^{m\times m}$ be symmetric, let $C\in\R^{n\times m}$, and let
	\begin{align}
	\Phi = \left[\begin{array}{lc}
	A & C\\
	C^T & B
	\end{array}\right].
	\end{align}
	Then $\Phi \succeq 0$ if and only if $A\succeq 0$, $AA^\dagger C = C$ and $B - C^TA^\dagger C \succeq 0$.
\end{theorem}

Let $t\in\R$ be a scalar slack variable. We are now ready to reformulate \eqref{eq:MLBLUE_MOSAP_1} as a SDP as follows:
\begin{align}
\label{eq:MLBLUE_MOSAP_SDP}
\min\limits_{t,\bm{n} \ge \bm{0}}\ t,\quad\text{s.t.}\quad \Phi(t,\bm{n}) = \left[\begin{array}{cc}
\Psi(\bm{n}) & \bm{e}_1\\
\bm{e}_1^T & t
\end{array}\right]\succeq 0,\quad \bm{n}^T \bm{c}\leq b,\quad \bm{n}^T \bm{h} \geq 1.
\end{align}
This is indeed an SDP since $\Psi(\bm{n})$ is linear in $\bm{n}$ and thus $\Phi(t,\bm{n})$ is linear in $(t,\bm{n})$.

\begin{theorem}
	\label{th:SDP_equivalence}
	Let Assumption \ref{ass:nonsingular_covariance} hold. Then problems \eqref{eq:MLBLUE_MOSAP_1} and \eqref{eq:MLBLUE_MOSAP_SDP} are equivalent.
\end{theorem}
\begin{proof}
	By Theorem \ref{th:Schur_pseudoinverse} we have that
	\begin{align}
	\label{eq:_thm_equivalence_MOSAP_SDP_1}
	\Phi\succeq 0\quad \Leftrightarrow\quad \Psi \succeq 0,\ \  \Psi\Psi^\dagger\bm{e}_1 = \bm{e}_1,\ \ t \geq \bm{e}_1^T\Psi^\dagger\bm{e}_1.
	\end{align}
	The first of the three conditions is automatically satisfied since $\Psi$ is non-negative definite. To complete the proof it is sufficient to prove that the second condition also always holds. In fact, this would imply that $\Phi\succeq 0$ if and only if $t\geq \bm{e}_1^T\Psi^\dagger\bm{e}_1$, and taking the minimum of the latter expression over $t$ would then yield the thesis, since $\min_{t} t = \bm{e}_1^T\Psi^\dagger\bm{e}_1$.
        Note that the spaces $\text{col}(\Psi)$ and $\text{row}(\Psi)$ coincide since $\Psi$ is symmetric. For the second condition on the right-hand side of \eqref{eq:_thm_equivalence_MOSAP_SDP_1} to hold it is sufficient that $\bm{e}_1\in \text{col}(\Psi)$ since $\Psi\Psi^\dagger$ is the orthogonal projector onto $\text{col}(\Psi)$. Note that since $\text{col}(\Psi)=\text{row}(\Psi)=\text{null}(\Psi)^\perp$, it is sufficient to demonstrate that
	$\bm{e}_1$ is orthogonal to $\text{null}(\Psi)$, which, by Theorem \ref{th:Psi_properties}, is given by $\text{null}(\Psi)=\bigcap_{n_k>0} \mathcal{V}_k$. We observe that the last constraint in \eqref{eq:MLBLUE_MOSAP_1} and \eqref{eq:MLBLUE_MOSAP_SDP} ensures that model $\mathcal{M}_1$ is sampled at least once. Hence, if we let $\bar{k}$ index the sampled groups (i.e., those for which $n_{\bar{k}}>0$) that include $\mathcal{M}_1$ we have
	\begin{align}
	\text{null}(\Psi) = \bigcap_{n_k>0} \mathcal{V}_k \subseteq\ \bigcap_{\bar{k}} \mathcal{V}_{\bar{k}} \subseteq \{\bm{v}\in\R^\ell\ :\ v_1 = 0 \} = \text{span}(\bm{e}_1)^\perp.
	\end{align}
	Hence $\bm{e}_1\in\text{null}(\Psi)^\perp=\text{row}(\Psi)=\text{col}(\Psi)$, and hence $\Psi\Psi^\dagger\bm{e}_1=\bm{e}_1$ and the theorem is proved.
\end{proof}

We have now proved that the MLBLUE MOSAP can be formulated as the SDP \eqref{eq:MLBLUE_MOSAP_SDP}, which can be solved reliably and efficiently. In practice it is also often useful to solve an alternative MOSAP in which the user prescribes a statistical error tolerance $\eps^2$ and minimizes the computational expense subject to this restriction. This new MOSAP reads
\begin{align}
\label{eq:MLBLUE_MOSAP_2}
\min\limits_{\bm{n} \ge \bm{0}} \, \bm{n}^T \bm{c},\quad \text{s.t.}\quad \bm{e}_1^T\Psi^{\dagger}(\bm{n})\bm{e}_1\leq \eps^2,\quad \bm{n}^T \bm{h} \geq 1,
\end{align}
which is a straightforward modification of \eqref{eq:MLBLUE_MOSAP_1}. Problem \eqref{eq:MLBLUE_MOSAP_2} appears harder to solve than \eqref{eq:MLBLUE_MOSAP_1} since the nonlinear variance function now appears as a constraint. Again, we can reformulate \eqref{eq:MLBLUE_MOSAP_2} as a SDP as follows
\begin{equation}
\label{eq:MLBLUE_MOSAP_SDP_2}
\min\limits_{\bm{n} \ge \bm{0}} \, \bm{n}^T \bm{c},\quad\text{s.t.}\quad \Phi(\eps^2,\bm{n}) = \left[\begin{array}{cc}
\Psi(\bm{n}) & \bm{e}_1\\
\bm{e}_1^T & \eps^2
\end{array}\right]\succeq 0,\quad \bm{n}^T \bm{h} \geq 1.
\end{equation}
\begin{corollary}
	Let Assumption \ref{ass:nonsingular_covariance} hold. Then problems \eqref{eq:MLBLUE_MOSAP_2} and \eqref{eq:MLBLUE_MOSAP_SDP_2} are equivalent.
\end{corollary}
\begin{proof}
	In the proof of Theorem \ref{th:SDP_equivalence} we have already shown that $\Phi(t,\bm{n})\succeq 0$ if and only if $\bm{e}_1^T\Psi^{\dagger}(\bm{n})\bm{e}_1\leq t$. Setting $t=\eps^2$ yields the thesis.
\end{proof}
Note that \eqref{eq:MLBLUE_MOSAP_SDP_2} is also a SDP and it is as easy to solve as
\eqref{eq:MLBLUE_MOSAP_SDP}. {In fact, the two problems are strongly related:
  by introducing a parameter
  $\tau>0$ we can trace out the Pareto frontier defined by the two
  objectives of statistical error and cost as follows:
\begin{equation}
\label{eq:MLBLUE_MOSAP_SDP_3}
\min\limits_{t, \bm{n} \ge \bm{0}} \, t + \tau \bm{n}^T \bm{c},\quad\text{s.t.}\quad \Phi(t,\bm{n}) = \left[\begin{array}{cc}
\Psi(\bm{n}) & \bm{e}_1\\
\bm{e}_1^T & t
\end{array}\right]\succeq 0,\quad \bm{n}^T \bm{h} \geq 1.
\end{equation}
There are values of $\tau$ that produce the solutions of
\eqref{eq:MLBLUE_MOSAP_SDP} and \eqref{eq:MLBLUE_MOSAP_SDP_2}. This
problem \eqref{eq:MLBLUE_MOSAP_SDP_3} is also a SDP
and can be solved as easily as the other two formulations.}

\begin{remark}
	It was shown in \cite{schaden2020multilevel} that some approximate control variate methods, such as the ACV-IS estimator of \cite{gorodetsky2020generalized}, are BLUEs for a specific set of allowed model groupings. However, it is nontrivial to solve the sample allocation problem in \cite{gorodetsky2020generalized} numerically. Formulations \eqref{eq:MLBLUE_MOSAP_SDP} and \eqref{eq:MLBLUE_MOSAP_SDP_2} can be used to accelerate and simplify the set up of {\em any} estimator that is also a MLBLUE.
\end{remark}

\begin{remark}
	The SDP formulations of MOSAP in \eqref{eq:MLBLUE_MOSAP_SDP}, \eqref{eq:MLBLUE_MOSAP_SDP_2}, and \eqref{eq:MLBLUE_MOSAP_SDP_3} can incorporate additional linear constraints (e.g., on the number of samples) while remaining SDPs. For instance, in Section \ref{sec:numerical_results} we consider a scenario in which the high-fidelity models are prohibitive to sample and the maximum number of high-fidelity samples is thus restricted.  
\end{remark}

An advantage of formulating the MOSAPs as SDPs is that they can easily
be extended to the multiple output case without impacting their
numerical solvability, as we show next.

\section{Multi-output multilevel best linear unbiased estimators}
\label{sec:multi-output}

In this section we extend MLBLUEs to the multi-output case and introduce SDP formulations for their MOSAPs. We tackle the multi-output estimation problem by simultaneously constructing $m$ MLBLUEs, one for each output, and setting up two multiple-output SDP MOSAPs, the first minimizing the total cost, and the second minimizing the maximum statistical error.

Consider a single model set $\mathcal{M} = \{\mathcal{M}_i\}_{i=1}^{\ell}$, corresponding now to $m$ different output sets indexed by $s=1,\dots,m$, so that $p_i^s$ is the output of the $i$-th model corresponding to the $s$-th QoI. We do not assume that all the outputs $p_i^s$ can actually be sampled, since models might not necessarily produce all outputs. For this reason, we construct $m$ separate model coupling sets $\mathcal{G}^s\subseteq 2^{\mathcal{M}}\setminus\emptyset$ for $s=1,\dots,m$, each containing the possible or feasible model groupings for the $s$-th QoI.

We now simultaneously construct $m$ MLBLUEs, each starting from its feasible set $\mathcal{G}^s$ obtaining the estimators
\begin{align}
\label{eq:MO_MLBLUE_HF}
\E[p_1^s]\approx \bm{e}_1^T\hat{\bm{\mu}}^s=\bm{e}_1^T\Psi^\dagger_s\hat{\bm{y}}_s,\quad \V[\bm{e}_1^T\hat{\bm{\mu}}^s] = \bm{e}_1^T\Psi^\dagger_s\bm{e}_1,\quad s=1,\dots,m,
\end{align}
where $\hat{\bm{y}}_s$, $\Psi_s$ are constructed in the same way as in the single-output case. We still consider each $\Psi_s$ to be of size $\ell$-by-$\ell$ even if not all $\ell$ models are used: if model $i$ cannot be used to estimate the $s$-th QoI, we simply set the $i$-th row and column of $\Psi_s$ to zero. With this strategy, we can incorporate models that output only some of the QoIs within the standard MLBLUE construction, so that optimality is still ensured. Note that a $\Psi_s$ constructed this way remains consistent with the properties of the matrix $\Psi$ discussed in Theorem~\ref{th:Psi_properties}.

To set up all $m$ estimators concurrently, we need to design and solve a multi-output MOSAP. We have three options: 1) fix a computational budget and minimize the maximum estimator variance. 2) Fix a statistical error tolerance $\eps^2_s$ for each output and minimize the total cost. 3) Trace out a Pareto frontier defined by the two objectives of maximum estimator variance and cost. The resulting formulations are as follows.

\paragraph{1) Minimizing the maximum estimator variance for a fixed budget.} We write the multi-output MOSAP as
\begin{align}
\label{eq:MLBLUE_MOMOSAP_std}
\min\limits_{t,\bm{n} \ge \bm{0}}\ \, t,\quad\text{s.t.}\quad \bm{e}_1^T\Psi_s^\dagger(\bm{n})\bm{e}_1\leq t,\quad \bm{n}^T \bm{h}^s \geq 1,\quad
\bm{n}^T \bm{c}\leq b,\quad s=1,\dots,m.
\end{align}
Its corresponding SDP reformulation reads:
\begin{align}
\label{eq:MLBLUE_MOMOSAP}
\min\limits_{t,\bm{n} \ge \bm{0}}\ \, t,\quad\text{s.t.}\quad \Phi_s(t,\bm{n}) = \left[\begin{array}{cc}
\Psi_s(\bm{n}) & \bm{e}_1\\
\bm{e}_1^T & t
\end{array}\right]\succeq 0,\quad \bm{n}^T \bm{h}^s \geq 1,\quad
\bm{n}^T \bm{c}\leq b,\quad s=1,\dots,m.
\end{align}
We now have $m$ different positive semidefiniteness constraints, one for each output, but otherwise this is similar to the single-output MOSAPs \eqref{eq:MLBLUE_MOSAP_1} and \eqref{eq:MLBLUE_MOSAP_SDP}. Let $\mathcal{G}=\bigcup_{s=1}^m\mathcal{G}^s$ be the set of all possible model groupings across all estimators. Then the vector $\bm{n}$ is of size $|\mathcal{G}|$ and the vectors $\bm{h}^s\in\R^{|\mathcal{G}|}$ are constructed in the same way as before: for all $\mathcal{G}_k\in\mathcal{G}$, set $\bm{h}_k=1$ if $\mathcal{M}_1\in \mathcal{G}_k\in\mathcal{G}^s$ and $\bm{h}_k=0$ otherwise. Following the same approach as in the proof of Theorem \ref{th:SDP_equivalence} it is easy to see that for any pair $(t,\bm{n})$ that is feasible for problem \eqref{eq:MLBLUE_MOMOSAP} we have that $\bm{e}_1^T\Psi_s^\dagger(\bm{n})\bm{e}_1\leq t$ for all $s$, and therefore the two optimization problems \eqref{eq:MLBLUE_MOMOSAP_std} and \eqref{eq:MLBLUE_MOMOSAP} are equivalent.

\begin{remark}
	Minimizing the maximum statistical error across all outputs as in \eqref{eq:MLBLUE_MOMOSAP} is not the only option and other objective functions can be considered. For instance, if we introduce a slack vector $\bm{t} \in \R^m$, with $t_s$ corresponding to the statistical error of output $s$, we can replace $\Phi_s(t,\bm{n})$ in \eqref{eq:MLBLUE_MOMOSAP} with $\Phi_s(\bm{t}_s,\bm{n})$, and then minimize a linear or convex quadratic function\footnote{A convex quadratic objective would add a second-order cone constraint to \eqref{eq:MLBLUE_MOMOSAP}; the resulting problem is still a conic program and is no more difficult to solve than the original SDP \cite{alizadeh2003second}. For example, to minimize $|| \bm{t} ||_2$, we introduce a scalar variable $\bar{t}$ and the constraint $\bar{t} \ge || \bm{t} ||_2$ (which can be expressed as a second-order cone constraint), and define the objective to be $\bar{t}$.} of $\bm{t}$, e.g., $||\bm{t}||_1$, $||\bm{t}||_{\bm{w}}=\sum w_s\bm{t}_s$ for some positive vector of weights $\bm{w}$, or $||\bm{t}||_2$.
\end{remark}

\paragraph{2) Minimizing the total cost for fixed statistical error tolerances.}
An alternative MOSAP formulation can be obtained by setting a separate statistical error tolerance $\eps_s^2>0$ for each output $s$, and then minimizing the total cost of all $m$ estimators. This yields the two equivalent formulations:
\begin{align}
\label{eq:MLBLUE_MOMOSAP_2_std}
\min\limits_{\bm{n}\ge \bm{0}} \, \bm{n}^T \bm{c},\quad\text{s.t.}\quad &\bm{e}_1^T\Psi_s^\dagger(\bm{n})\bm{e}_1\leq \eps^2_s,\qquad\qquad\qquad\qquad\ \  \bm{n}^T \bm{h}^s \geq 1,\quad s=1,\dots,m,\\
\label{eq:MLBLUE_MOMOSAP_2}
\min\limits_{\bm{n}\ge \bm{0}} \, \bm{n}^T \bm{c},\quad\text{s.t.}\quad &\Phi_s(\eps^2_s,\bm{n}) = \left[\begin{array}{cc}
\Psi_s(\bm{n}) & \bm{e}_1\\
\bm{e}_1^T & \eps^2_s
\end{array}\right]\succeq 0,\quad\ \ \  \bm{n}^T \bm{h}^s \geq 1,\quad s=1,\dots,m.
\end{align}
The formulations \eqref{eq:MLBLUE_MOMOSAP_2_std} and \eqref{eq:MLBLUE_MOMOSAP_2} extend \eqref{eq:MLBLUE_MOSAP_2} and \eqref{eq:MLBLUE_MOSAP_SDP_2}, respectively, to the multi-output case.

\paragraph{3) Trace out the Pareto frontier between the two objectives of maximum estimator variance and total cost.}
A third possibility is to introduce a nonnegative parameter $\tau$ and consider the MOSAP
\begin{align}
\label{eq:MLBLUE_MOMOSAP_std_Pareto}
\min\limits_{t,\bm{n} \ge \bm{0}}\ \, t + \tau \bm{n}^T \bm{c} ,\quad\text{s.t.}\quad \bm{e}_1^T\Psi_s^\dagger(\bm{n})\bm{e}_1\leq t,\quad \bm{n}^T \bm{h}^s \geq 1,\quad s=1,\dots,m.
\end{align}
Its corresponding SDP reformulation reads:
\begin{align}
\label{eq:MLBLUE_MOMOSAP_Pareto}
\min\limits_{t,\bm{n} \ge \bm{0}}\ \, t + \tau \bm{n}^T \bm{c},\quad\text{s.t.}\quad \Phi_s(t,\bm{n}) = \left[\begin{array}{cc}
\Psi_s(\bm{n}) & \bm{e}_1\\
\bm{e}_1^T & t
\end{array}\right]\succeq 0,\quad \bm{n}^T \bm{h}^s \geq 1,\quad s=1,\dots,m.
\end{align}
Equation \eqref{eq:MLBLUE_MOMOSAP_Pareto} is the multi-output equivalent of \eqref{eq:MLBLUE_MOSAP_SDP_3}. Its solution will match that of \eqref{eq:MLBLUE_MOMOSAP} for some nonnegative value of $\tau$. For the variant of \eqref{eq:MLBLUE_MOMOSAP_2} in which we define  $\varepsilon_s$ to have the same positive value for all $s=1,2,\dots,m$,  there is also a nonnegative value of $\tau$ for which the solutions of \eqref{eq:MLBLUE_MOMOSAP_Pareto} and \eqref{eq:MLBLUE_MOMOSAP_2} coincide.

\paragraph{}
We highlight the fact that these multi-output MLBLUEs leave the model selection completely to the MOSAP \emph{without need for the user to make any prior decision about the quality of the models or the structure of the model set}. This is a convenient feature which becomes essential when models are heterogeneous and there is no clear distinction as to which ones are more suitable for the joint estimation problem. We also remark that multi-output MLBLUEs are the first instance of a multi-output multilevel or multifidelity estimator that simultaneously constructs a separate estimator for each QoI.

\begin{remark}
	Note that the MLBLUEs in \eqref{eq:MO_MLBLUE_HF} are all BLUE if considered \emph{independently}, but they are not BLUE if the multi-output estimation is considered \emph{jointly}. One could account for the correlations between different outputs and construct a single multi-output MLBLUE in which different outputs are treated as different ``models''. However, this would lead to a total of up to $m\ell$ possible models, one for each (model, output) combination, causing the number of possible groupings to scale like $O(2^{m\ell})$ and considerably increase the cost of solving the MOSAP.
\end{remark}

\begin{remark}
	We remark that single and multi-output MLBLUEs (and all multifidelity methods based on covariances) can be extended to handle infinite-dimensional QoI with minor modifications. Provided that the output belongs to a Bochner space $\mathcal{L}^2(\Omega, \mathcal{V})$ for a suitable sample space $\Omega$ and a Hilbert space $\mathcal{V}$, it is sufficient to define the covariance operator (and consequently variance and correlation operators) as the inner product $(\cdot,\cdot)_{\mathcal{L}^2(\Omega, \mathcal{V})}$ of $\mathcal{L}^2(\Omega, \mathcal{V})$, i.e.,
	\begin{align}
	\label{eq:covariance_hilbert}
	C(X,Y) = (X,Y)_{\mathcal{L}^2(\Omega, \mathcal{V})} = \E[(X-\E[X],\ Y - \E[Y])_{\mathcal{V}}],\quad\forall X,Y\in \mathcal{L}^2(\Omega, \mathcal{V}),
	\end{align}	
	 where $(\cdot,\cdot)_{\mathcal{V}}$ is the inner product in $\mathcal{V}$. The above is a natural generalization: by taking $\mathcal{V}=\R$ we recover the standard covariance operator between scalar random variables. This approach has the advantage of yielding scalar-valued covariances that can be used in multifidelity methods in the same way as in the scalar case.
\end{remark}

\section{Numerical results}
\label{sec:numerical_results}

In this section we test both the new SDP formulations of the MLBLUE MOSAP and the single- and multi-output MLBLUEs. The aim is to demonstrate the advantages of the new MOSAP formulations and of MLBLUEs themselves in problems that reflect the complexity of typical computational engineering applications. Specifically, we will show how the improved efficiency of MLBLUEs results from their ability to work with any selection of model groupings. Here we only compare MLBLUE with MLMC and MFMC: performance comparisons between MLBLUE and other estimators such as the ACV methods for a single QoI are available in \cite{schaden2020multilevel,schaden2021asymptotic} and extending other estimators to multi-output problems is beyond the scope of this work.

We implemented the single- and multi-output MLBLUEs (as well as single- and multi-output MLMC and MFMC) into an open-source Python software, \texttt{BLUEST}, which is publicly available on GitHub at \url{https://github.com/croci/bluest}. \texttt{BLUEST} allows users to easily wrap Python models (and, by extension, any model that can be called from Python), and supports MPI sample parallelization of serial and MPI-parallelized models.

\subsection{Model heterogeneity}
We start by considering two problems with heterogeneous model sets: Problem 1 uses global and local grid refinement to define the models while Problem 2 uses global grid and timestep refinement as well as a combination of different PDE and ODE models. The objective here is to demonstrate the improved efficiency of multi-output MLBLUEs and of the new SDP formulations. We first investigate the computational efficiency of the MLBLUE, MLMC, and MFMC methods constructed for each problem and then describe the computational advantages of the new MOSAP formulations.

\subsubsection{Problem 1: Steady Navier-Stokes flow past two cylinders}
\label{subsubsec:NS}

\begin{figure}[h!]
	\centering
	\includegraphics[width=0.7\textwidth]{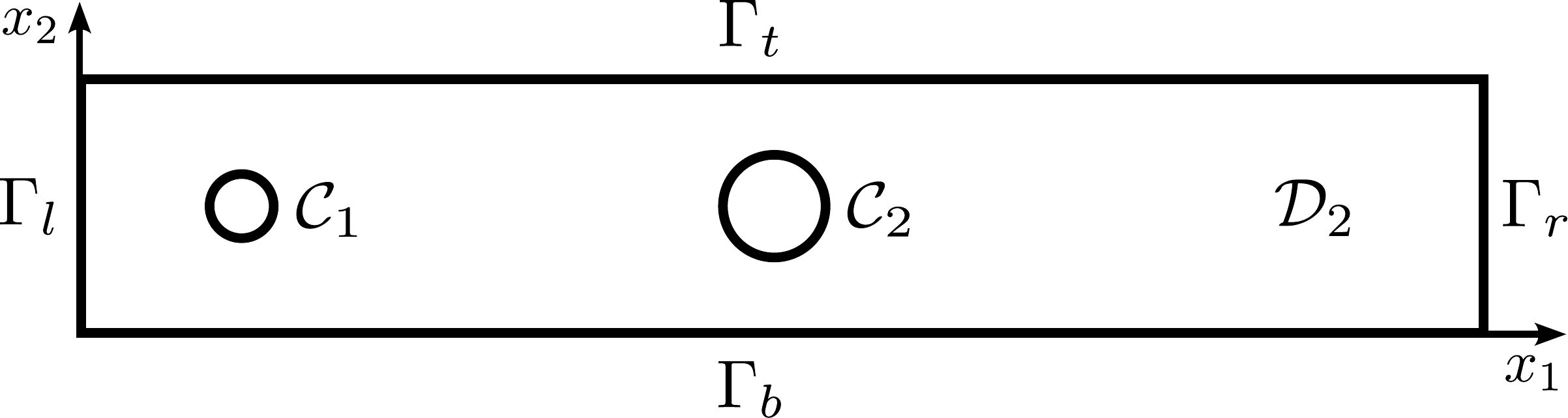}
	\caption{\textit{The domain of Problem 1. The rectangle has length $|\Gamma_b|=|\Gamma_t|=2.2$ and width $|\Gamma_l|=|\Gamma_r|=0.41$. The two obstacles $\mathcal{C}_1$ and $\mathcal{C}_2$ have centers at $(0.2,0.2)$ and $(1.0,0.2)$ and radii $0.05$ and $0.08$ respectively. $\Gamma_l$ and $\Gamma_r$ are the inflow and outflow boundaries respectively and the other boundaries are walls.}}
	\label{fig:domain_NS}
	\vspace{0pt}
\end{figure}

\paragraph{Problem description.} As a first test problem we consider a steady flow past two cylindrical obstacles of different sizes, as described by the steady Navier-Stokes equations:
\begin{align}
\label{eq:problem2}
\begin{dcases}
\begin{array}{ll}
-\nu(\omega)\Delta \bm{u} + \bm{u}\nabla \bm{u} + \nabla q = 0,\quad \nabla\cdot \bm{u} = 0, & \bm{x}\in \mathcal{D}_2,\\
\bm{u}|_{\Gamma_t}=\bm{u}|_{\Gamma_b}=\bm{u}|_{\partial \mathcal{C}_1}=\bm{u}|_{\partial \mathcal{C}_2}= \bm{0}, & \nu(\omega)\nabla\bm{u}|_{\Gamma_r}\bm{\eta}-q|_{\Gamma_r}\bm{\eta} = 0,\\
\bm{u}|_{\Gamma_l} = \left({4|\Gamma_l|^{-2}u_{\text{in}}(\omega)x_2(|\Gamma_l|-x_2)},\ 0\right)^T, & \bm{x}=(x_1,x_2)^T,
\end{array}
\end{dcases}
\end{align}
where the domain $\mathcal{D}_2$ is described Figure~\ref{fig:domain_NS}. Here $\bm{u}$ is the flow velocity and $q$ its pressure, and $\bm{\eta}$ denotes the outer unit normal vector of the boundary. The viscosity $\nu$ and the maximum inlet velocity $u_{\text{in}}$ are modeled as uniform random variables, with $\nu(\omega)\in[5\times 10^{-4},\ 1.5\times 10^{-3}]$ and $u_{\text{in}}\in[0.1,\ 0.5]$. This leads to a Reynolds number which is also random (albeit not uniform), and given by
\begin{align*}
\text{Re}(\omega)=\frac{u_{\text{mean}}(\omega)}{\nu(\omega)}\frac{|\mathcal{C}_2|}{\pi},\quad \text{Re}(\omega)\in [7.\bar{1},\ 106.\bar{6}].
\end{align*}
where $u_{\text{mean}}=\frac{2}{3}u_{\text{in}}$ is the mean inlet velocity and $|\mathcal{C}_2|/\pi$ is the diameter of the largest obstacle.

\paragraph{Quantities of interest.}
We consider six QoIs: the drag and lift coefficients around each obstacle
\begin{align}
(c_{\text{drag}}^i,c_{\text{lift}}^i)^T=\frac{2\pi}{u_{\text{mean}}^2|\mathcal{C}_i|}\int_{\partial \mathcal{C}_i}(\nu\nabla\bm{u}-pI)\bm{\eta}\text{ d}s,\quad i=1,2,
\end{align}
as well as the difference in pressure before and after each cylinder
\begin{align}
\Delta_{\mathcal{C}_1} q = q(0.15,0.2)-q(0.25,0.2),\quad \Delta_{\mathcal{C}_2} q = q(0.72,0.2)-q(0.88,0.2).
\end{align}

\paragraph{Model set.} We construct twelve grid-based models for this problem as follows. First, we construct three base meshes, one coarse, one medium, and one fine. Then we locally refine each base mesh in three different ways: around both obstacles, around $\mathcal{C}_1$ only, or around $\mathcal{C}_2$ only. We then have three base grids with four configurations each (three types of local refinements $+$ no refinement), for a total of twelve grids, which we show in Figure \ref{fig:NS_grids}. We build the meshes with the open-source software \texttt{Gmsh}\footnote{Available at \url{https://gmsh.info/}.} (version $4.10.5$) and the mesh-generating scripts are available in the \texttt{BLUEST} open-source repository\footnote{See \texttt{mesh\_generator.py} in the \texttt{bluest/examples/navier-stokes} folder in the \texttt{BLUEST} repository.}. Note that the grids refined around a single obstacle improve the approximation of the QoIs defined on that obstacle, but not of those defined on the other cylinder. We design the model set in this way to investigate the performance of multi-output MLBLUEs in a setting in which some low-fidelity models yield good approximations of only some QoIs. While Problem~1 is a simple example, a conceptually similar setting often arises in multi-physics problems where low-fidelity models are defined by simplifying some of the physics. In such problems, the low-fidelity models will yield poor approximations of the QoIs related to the physics that were simplified.

\begin{figure}[h!]
	\centering
	\includegraphics[width=\textwidth]{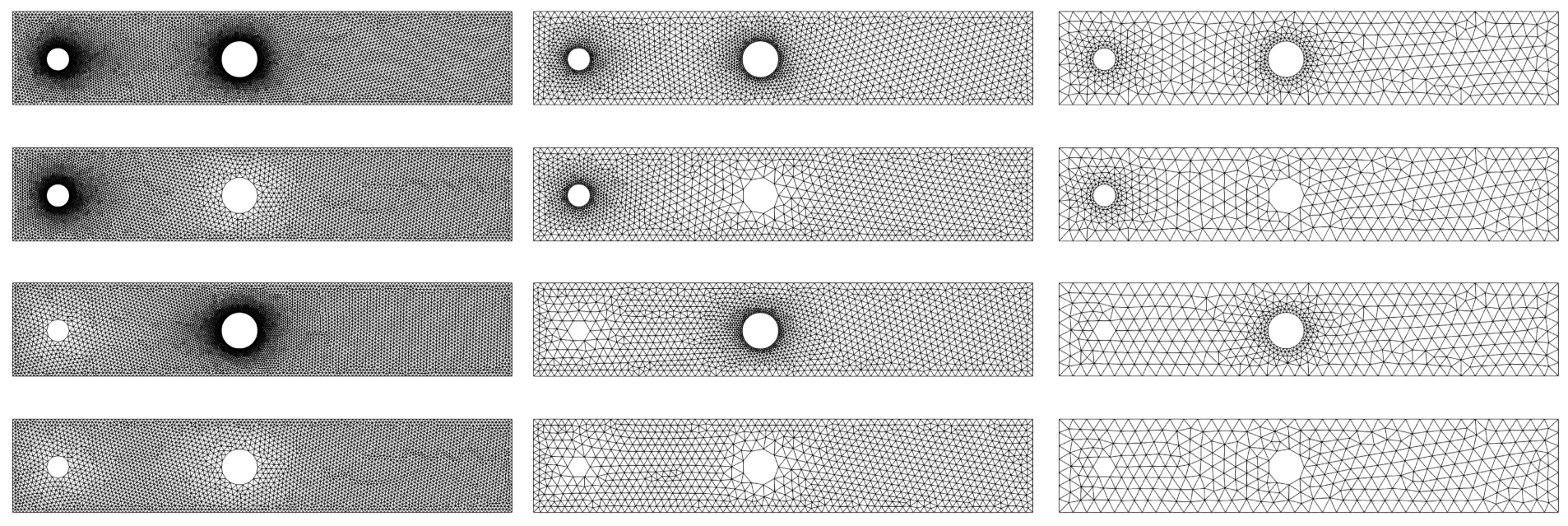}
	\caption{\textit{The grids used to define the twelve models for Problem 1. Each column corresponds to a different ``base'' mesh: fine (left), medium (center), coarse (right). The actual ``base'' meshes are in the last row.}}
	\label{fig:NS_grids}
	\vspace{0pt}
\end{figure}

\paragraph{Covariance estimation.} We take $100$ pilot samples of all models to estimate model covariances. This is a larger amount than needed in practice in most cases \cite{NgWillcox2014multifidelity}; we choose a high number to sanitize our results from pilot sampling errors.

\paragraph{Solver and model costs.} We solve \eqref{eq:problem2} over all grids with the FE method using the Taylor-Hood FE pair and the \texttt{FEniCS} library \cite{LoggEtAl2012} using a direct linear solver. As to costs, we use the number of degrees of freedom as a proxy for simplicity: we set these pseudo-costs to be $(n_{\text{dofs}}^i)^2$, where $n_{\text{dofs}}^i$ is the number of degrees of freedom of model $i$, and we have assumed a quadratic cost complexity for the nonlinear solver used to solve \eqref{eq:problem2}. Going from top to bottom and from left to right in Figure \ref{fig:NS_grids}, the number of degrees of freedom for each grid are $\{68619, 47952, 53817, 33114\}$ (fine grids), $\{18512, 12875, 14471, 8924\}$ (medium grids) $\{5225, 3932, 4274, 3026\}$ (coarse grids). We do not account for pilot samples in our cost computations since all the methods we compare (MLBLUE, MLMC and MFMC) benefit from more accurate estimates.

\paragraph{Estimator setup.}
We solve the statistical-error-constrained SDP MOSAP \eqref{eq:MLBLUE_MOMOSAP_2} with tolerances $\varepsilon_s=10^{-3}\sqrt{\V[p^s_1]}$ for $s=1,\dots m=6$, where $\V[p^s_1]$ is estimated from the pilot samples. Here we consider all model groupings of size at most $\kappa=7$ (cf.~Remark \ref{rem:max_group_size}). For constructing the MLMC and MFMC estimators, we proceed as follows. For each possible model combination that includes the high-fidelity model, we loop over all QoIs $s=1,\dots, m$. For each QoI, we order the models (by cost for MLMC \cite{giles2015multilevel} and by correlation for MFMC \cite{NgWillcox2014multifidelity}) and compute the QoI optimal sample allocation satisfying the tolerance $\varepsilon_s$ by using the MLMC/MFMC single-output analytic expression. We then take the maximum number of samples across all QoIs and use the result to compute the corresponding total multi-output estimator cost for that model combination. Finally, we select the model group corresponding to the minimum total cost. Note that by setting up MLMC and MFMC this way there is no need to provide any model ordering \emph{a priori}. For MFMC, if for any model combination there is a QoI for which the MFMC model ordering constraints \cite{NgWillcox2014multifidelity} are not satisfied, then that model combination is discarded. (Note that MLMC does not have an equivalent restriction on the variances\footnote{MLMC variances are often assumed as decaying in MLMC convergence theory. However, this is not a requirement of MLMC itself \cite{giles2015multilevel}.} \cite{giles2015multilevel}.) In all methods, once a real-valued sample allocation $\bm{n}^*$ is found, we project it to an integer-valued solution (cf.~Remark \ref{rem:integer_proj}).

\begin{figure}[h!]
	\centering
	\includegraphics[width=\textwidth]{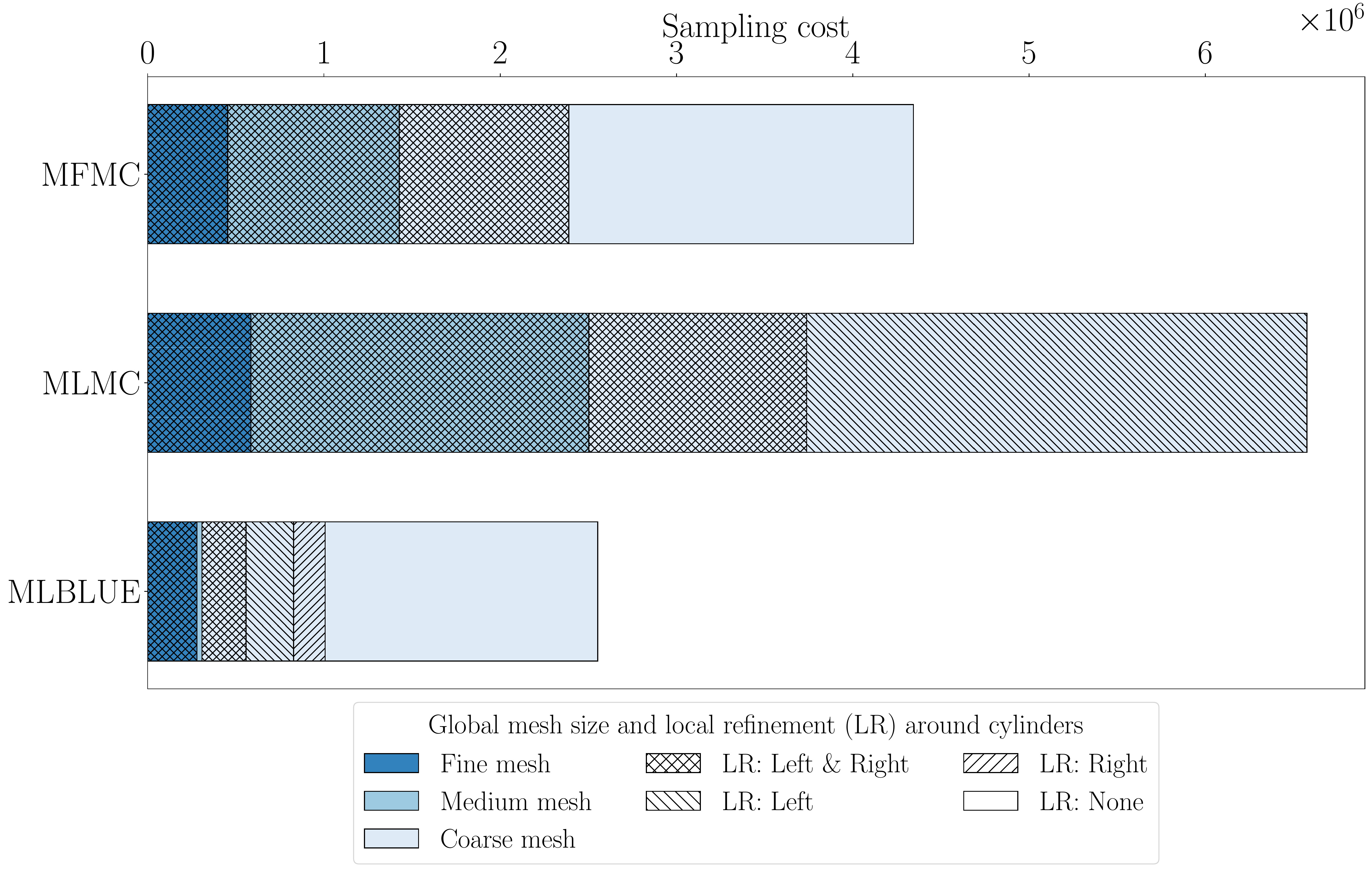}
	\caption{\textit{Total cost of the MOSAP-optimal MLBLUE, MLMC, and MFMC methods for Problem 1. Colors denote the ``base'' meshes (darker $=$ finer), while the patterns indicate the local refinements around obstacles.}}
	\label{fig:NS_cost_complexity}
	\vspace{0pt}
\end{figure}

\begin{figure}[h!]
	\centering
	\includegraphics[width=0.8\textwidth]{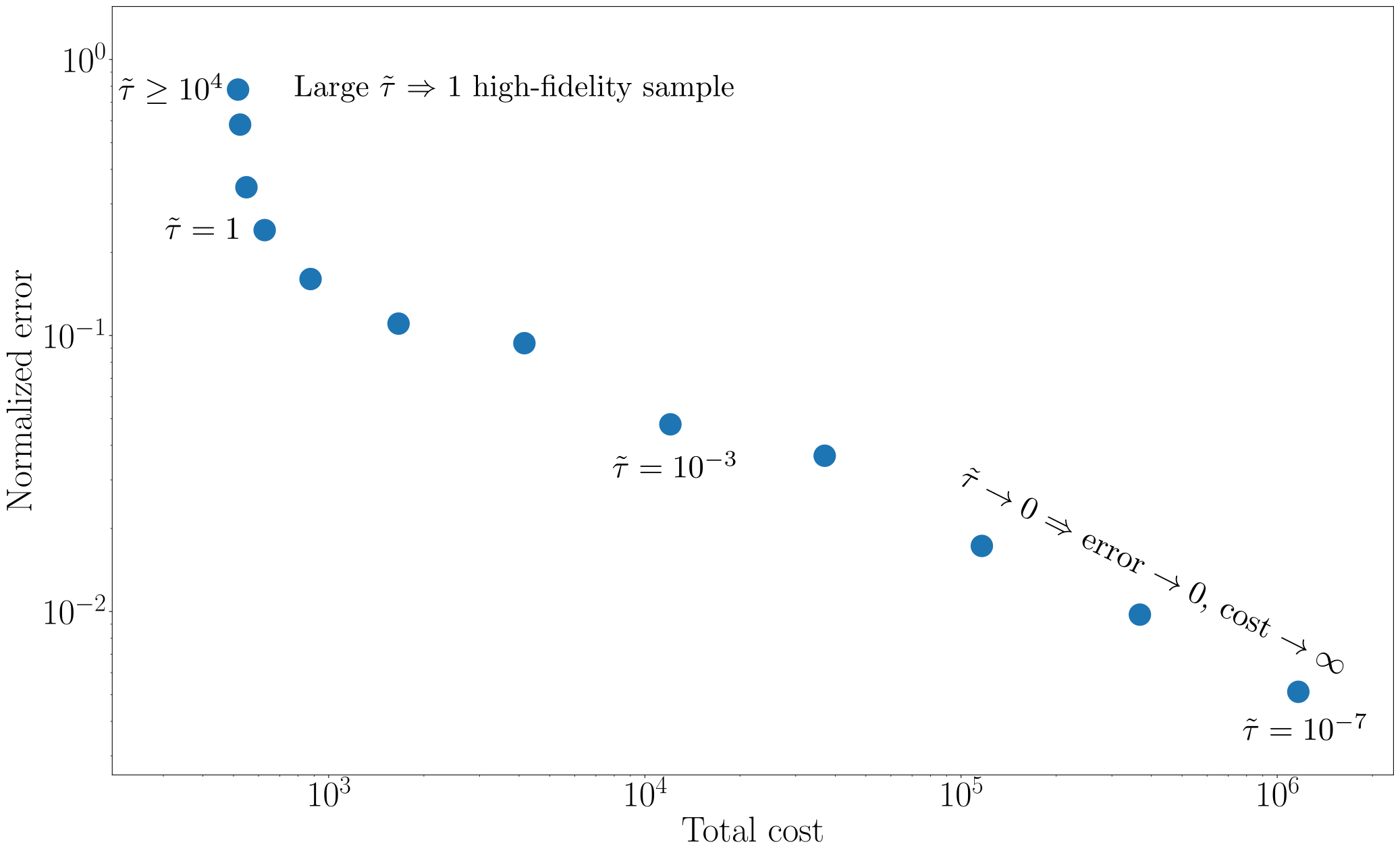}
	\caption{\textit{Points on the Pareto frontier MLBLUE for Problem 1. Here $\tilde{\tau}=\tau\lVert \bm{c} \rVert_2$, where $\tau$ is the Pareto parameter from \eqref{eq:MLBLUE_MOMOSAP_Pareto} and $\bm{c}$ is the model cost vector. From right to left, the points in the figure correspond to $\tilde{\tau}=10^i$ for $i=-7,\dots,4$.}}
	\label{fig:pareto_front_NS}
	\vspace{0pt}
\end{figure}

\paragraph{Results: total estimation cost.}
Once the optimal MLBLUE, MLMC and MFMC estimators have been constructed, we compute their total cost so that we can compare their efficiency. Results are shown in Figure \ref{fig:NS_cost_complexity}. MLBLUE is roughly three times more efficient than MLMC and two times more efficient than MFMC for the same statistical error tolerances.
The higher efficiency of MLBLUE derives from a better management and exploitation of the cheaper local refinement models: MLBLUE exploits all coarse grid models and only takes a few samples of the expensive medium grid models. In contrast, MLMC and MFMC draw many medium grid model samples, which account for a considerable portion of their total cost.

\paragraph{Results: Pareto frontier optimization.}
We solve the Pareto formulation of the MLBLUE MOSAP \eqref{eq:MLBLUE_MOMOSAP_Pareto} to analyze the trade-off between the two constraints of budget and accuracy. Figure \ref{fig:pareto_front_NS} shows the Pareto-optimal points corresponding to different values of the Pareto parameter $\tau$. We take $\tau=\tilde{\tau}\lVert \bm{c} \rVert_2^{-1}$, where $\tilde{\tau}=10^i$ for $i=-7,\dots,4$. The points are plotted in the total MLBLUE cost vs normalized error plane, where
\begin{align}
\text{normalized error} = \max_{s}\left(\frac{\V[\bm{e}^T_1\hat{\bm{\mu}}^s]}{\V[p^s_1]}\right)^{1/2}.
\end{align}
For $s=1,\dots,m$, the high-fidelity model output variances $\V[p^s_1]$ are estimated from pilot samples and $\V[\bm{e}^T_1\hat{\bm{\mu}}^s]$ is the estimator variance computed from the MOSAP solution using \eqref{eq:MO_MLBLUE_HF}. For large enough $\tau$, accuracy becomes unimportant and the MOSAP yields the cheapest, least accurate estimator satisfying the constraints in \eqref{eq:MLBLUE_MOMOSAP_Pareto}, namely, a single Monte Carlo sample of the high-fidelity model. On the other hand, as $\tau\rightarrow 0$, the Pareto frontier becomes a line in the loglog plot with a slope of $-\frac{1}{2}$ as expected for a Monte Carlo simulation: error $=$ cost$^{-1/2}$. For the values of $\tau$ between these two extremes, the behavior of the frontier is less obvious, and in this low-to-moderate-budget regime, it provides precise quantative information about the available tradeoffs. Here we have aggregated the MLBLUE error across all QoIs for simplicity. For more precise calibration, it may be beneficial to compute separate Pareto frontiers for each output. 

\subsubsection{Problem 2: Hodgkin-Huxley model}

\paragraph{Problem description - Hodgkin-Huxley high-fidelity model.} As a second test problem we consider the Hodgkin-Huxley equations for neuron membrane action potential in 1D on the unit interval:
\begin{align}
\label{eq:HH}
\begin{dcases}
c(\omega)v_t &= \iota(\omega) + \epsilon(\omega)\Delta v + g_{\text{K}}\alpha^4(v_{\text{K}}-v) + g_{\text{Na}}\beta^3 \gamma(v_{\text{Na}}-v) + g_l(v_l - v),\quad t\in[0,20].\\
\alpha_t &= \phi_\alpha(v)(1-\alpha)-\psi_\alpha(v)\alpha,\\
\beta_t &= \phi_\beta(v)(1-\beta)-\psi_\beta(v)\beta,\\
\gamma_t &= \phi_\gamma(v)(1-\gamma)-\psi_\gamma(v)\gamma,
\end{dcases}
\end{align}
\vspace{-12pt}
\begin{align}
%\text{where}\qquad\qquad
\begin{array}{ll}
\phi_\alpha(v) = \dfrac{1}{10}\dfrac{1-v/10}{e^{1-v/10}-1}, & \psi_\alpha(v) = \dfrac{e^{-v/80}}{8},\\[9pt]
\phi_\beta(v) = \dfrac{2.5-v/10}{e^{2.5-v/10}-1}, & \psi_\beta(v) = 4e^{-v/18},\\[9pt]
\phi_\gamma(v) = 0.07e^{-v/20}, & \psi_\gamma(v) = (1+e^{3-v/10})^{-1}.
\end{array}
\end{align}
The boundary conditions are $v_x(t,0,\omega) = v(t,1,\omega)=0$ for all $\omega,t$ (homogeneous Neumann and Dirichlet on the left and right boundary respectively), and the initial conditions are given by
\begin{align}
\label{eq:HH_BCs}
v(0,x,\omega)=v_{\text{eq}},\quad w(0,x,\omega)=\frac{\phi_w(v-v_{\text{eq}})}{\phi_w(v-v_{\text{eq}}) + \psi_w(v-v_{\text{eq}})},\quad\text{for}\quad w\in\{\alpha,\beta,\gamma\},
\end{align}
where $v_{\text{eq}}$ is the equilibrium potential, whose value is given in Table \ref{tab:HH_params} together with the values of all the deterministic parameters used for this problem. We refer the reader to \cite{chapman2021introduction} for a derivation and description of the original Hodgkin-Huxley model in ordinary differential equation (ODE) form and its PDE extension, which was studied in \cite{lieberstein1967hodgkin}.

\begin{table}[h!]
	\centering
	\begin{tabular}{@{}cll@{}}
		\toprule
		\multicolumn{1}{c}{\textbf{Parameter}}       & \multicolumn{1}{c}{\textbf{Meaning}}                      & \multicolumn{1}{c}{\textbf{Value}}             \\ \midrule
		$g_{\text{Na}}$ & Sodium conductance p.u.a.       & $120$ mScm$^{-2}$  \\
		$g_{\text{K}}$  & Potassium conductance p.u.a.    & $36$ mScm$^{-2}$   \\
		$g_l$           & Leakage conductance p.u.a.      & $0.3$ mScm$^{-2}$  \\
		$v_{\text{Na}}$ & Sodium Nernst potential         & $56$ mV            \\
		$v_{\text{K}}$  & Potassium Nernst potential      & $-77$ mV           \\
		$v_l$           & Leakage Nernst potential        & $-60$ mV           \\
		$v_{\text{eq}}$ & Equilibrium potential           & $-67.38614$ mV           \\
		$c_0$           & Membrane capacitance p.u.a.     & $1\ \mu$Fcm$^{-2}$  \\
		$\epsilon_0$    & Neuron fiber conductance        & $0.33616$ mS       \\
		$\iota_0$           & Applied current p.u.a.          & $33.6$ mAcm$^{-2}$ \\ \bottomrule
	\end{tabular}
	\caption{\textit{Parameters used in the Hodgkin-Huxley model \eqref{eq:HH} and their values. The units S, V, F, and A indicate siemens, volts, farads and amperes respectively. The acronym ``p.u.a.'' stands for ``per unit area''.}}
	\label{tab:HH_params}
\end{table}

In \eqref{eq:HH}, $v=v(t,x,\omega)$ is the action potential across a neuron membrane, which depends on the difference in the concentration of sodium (Na) and potassium (K) ions inside and outside of the neuron cell. These differences in concentrations are maintained by active transport mechanisms such as sodium-potassium pumps. In \eqref{eq:HH}, these mechanisms are modeled by the variables $\alpha=\alpha(t,x,\omega)$, $\beta=\beta(t,x,\omega)$, $\gamma=\gamma(t,x,\omega)$, with $0\leq \alpha,\beta,\gamma \leq 1$, which represent the activation and inactivation of gated ion transport channels. The variable $\alpha$ is the potassium gated channel activation, and the variables $\beta$ and $\gamma$ are the sodium gated channel activation and inactivation respectively. Additionally, we model the membrane capacitance $c(\omega)$, the neuron fiber conductance $\epsilon(\omega)$ and the applied current per unit area $\iota(\omega)$ as random variables defined by:
\begin{align}
c(\omega) = (c_0 + 0.2z_1^2(\omega))\ \mu\text{Fcm}^{-2},\quad \epsilon(\omega)=e^{z_2}\ \text{mS},\quad \iota(\omega) = \iota_0 \times 5^{2\upsilon(\omega)-1}\ \text{mAcm}^{-2},
\end{align}
where $z_1(\omega)$ and $\upsilon(\omega)$ are standard Gaussian and uniform random variables respectively, and $z_2(\omega)$ is a Gaussian random variable with mean and variance chosen so that $\E[e^{z_2}]=\epsilon_0$ and $\V[e^{z_2}]^{1/2}=0.2\epsilon_0$.

\paragraph{Problem description - FitzHugh-Nagumo low-fidelity model.} As a low-fidelity model, we take the FitzHugh-Nagumo model, which is obtained by introducing the simplifying assumptions that $\beta$ and $\alpha+\gamma$ are constant in time and space with $\beta=\bar{\beta}$ and $\alpha+\gamma = \bar{\gamma}$. The values of $\bar{\beta}$ and $\bar{\gamma}$ are obtained from the initial conditions \eqref{eq:HH_BCs}. The resulting FitzHugh-Nagumo equations read
\begin{align}
\label{eq:FN}
\begin{dcases}
c(\omega)v_t &= \iota(\omega) + \epsilon(\omega)\Delta v + g_{\text{K}}\alpha^4(v_{\text{K}}-v) + g_{\text{Na}}\bar{\beta}^3 (\bar{\gamma}-\alpha)(v_{\text{Na}}-v) + g_l(v_l - v),\\
\alpha_t &= \phi_\alpha(v)(1-\alpha)-\psi_\alpha(v)\alpha,
\end{dcases}
\end{align}
where the initial and boundary conditions for $v$ and $\alpha$ are the same as for the Hodgkin-Huxley model. We refer the reader to \cite{chapman2021introduction} for further information about this model.

\paragraph{Quantities of interest.} As QoIs, we choose the peak membrane potential $v_{\text{peak}}(\omega)=\max_{t,x}v(t,x,\omega)$, and the net total membrane, potassium, sodium, and leakage currents, which are respectively given by:
\begin{align}
\iota_{\text{me}}(\omega)&=\int_{0}^{20}\int_0^1\left(\iota(\omega) + \epsilon(\omega)\Delta v\right)\text{ d}x\text{ dt},\quad\iota_{\text{Na}}(\omega)=\int_{0}^{20}\int_0^1g_{\text{Na}}\beta^3 d(v_{\text{Na}}-v)\text{ d}x\text{ dt},\\
\iota_{\text{K}}(\omega)&=\int_{0}^{20}\int_0^1g_{\text{K}}\alpha^4(v_{\text{K}}-v)\text{ d}x\text{ dt},\quad\quad\quad\iota_{l}(\omega)=\int_{0}^{20}\int_0^1g_l(v_l-v)\text{ d}x\text{ dt}.
\end{align}

\paragraph{Model set.} We consider four base models: the Hodgkin-Huxley PDE model \eqref{eq:HH}, the FitzHugh-Nagumo PDE model \eqref{eq:FN}, and the corresponding ODE models obtained by setting $\epsilon=0$ and taking $v$, $\alpha$, $\beta$, and $\gamma$ to be constant in space. From each of these models, we consider three different types of grid and timestep refinements: we discretize the equations in time and space (PDE models only) using uniform grids with mesh size and timestep respectively given by $\Delta x=2^{-(2i+1)}$ and $\Delta t=\frac{2}{5}\Delta x$ for $i=1,2,3$. This yields a model set comprised of twelve models in total, in which the Hodgkin-Huxley PDE model with the finest grid and timestep is the high-fidelity model. This model set is an example of a heterogeneous model set in which again not all models are good for all outputs. For instance, the ODE models are unlikely to accurately approximate $\iota_{\text{me}}$ (the diffusion term is zero for the ODEs), and the FitzHugh-Nagumo models are unlikely to yield accurate samples of $\iota_{\text{Na}}$ (it involves quantities approximated as constants).

\paragraph{Covariance estimation.} As for Problem 1, we take $100$ pilot samples of all models to estimate model correlations.

\paragraph{Solver and model costs.} We discretize the PDEs in space with the FE method using piecewise linear elements in \texttt{FEniCS} \cite{LoggEtAl2012} and both ODEs and PDEs in time using the backward Euler method.  We set the model costs to be given by $\tilde{c}_i n_{\text{ts}}^i n_{\text{dofs}}^i$, where $n_{\text{ts}}^i$ and $n_{\text{dofs}}^i$ are the number of timesteps and of degrees of freedom used by model $i$ respectively ($n_{\text{dofs}}^i$ is the number of equations in the ODE models). Here the $\tilde{c}_i$ factor is $1$ for the ODE models and $8$ for the PDE models, and accounts for the fact that PDEs have a slightly higher cost per degree-of-freedom than ODEs. The values of $\tilde{c}_i$ were estimated from CPU times.

\begin{figure}[h!]
	\centering
	\includegraphics[width=\textwidth]{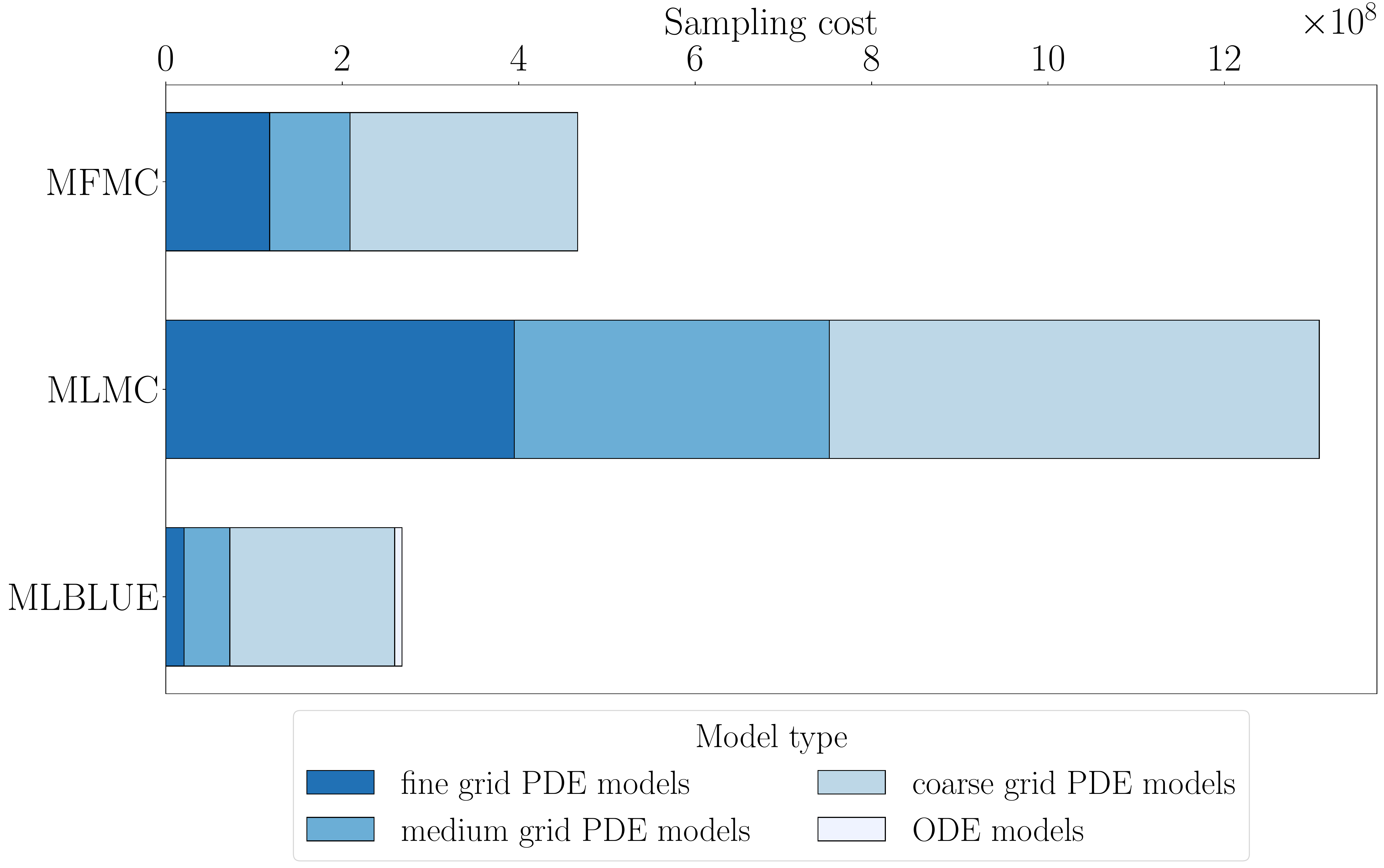}
	\caption{\textit{Total cost of the MOSAP-optimal MLBLUE, MLMC, and MFMC methods for Problem~2. The three darkest colors correspond to the PDE models with darker tints indicating finer grids and timesteps. The lightest color corresponds to the total computational effort spent on all the ODE models. Here we do not distinguish the use of the Hodgkin-Huxley versus the FitzHugh-Nagumo models for the sake of clarity, but we present the detailed sample data in Table \ref{tab:HH_number_of_samples}.}}
	\label{fig:HH_cost_complexity}
	\vspace{-12pt}
\end{figure}

\begin{table}[]
	\centering
	\resizebox{\textwidth}{!}{%
		\begin{tabular}{@{}l|rrr|c|ccc|ccc@{}}
			\toprule
			& \multicolumn{3}{c|}{H-H PDE}   & F-N PDE & \multicolumn{3}{c|}{H-H ODE} & \multicolumn{3}{c}{F-N ODE}   \\
			& fine    & med.     & coarse    & coarse  & fine   & med.    & coarse    & fine    & med.    & coarse    \\ \midrule
			MLMC   & $11959$ & $169150$ & $3856682$ & $0$     & $0$    & $0$     & $0$       & $0$     & $0$     & $0$       \\
			MFMC   & $3561$  & $43181$  & $1792145$ & $0$     & $0$    & $0$     & $0$       & $0$     & $0$     & $0$       \\
			MLBLUE & $626$   & $24527$  & $1249805$ & $95652$ & $938$  & $6441$  & $259634$  & $69746$ & $69121$ & $6356360$ \\ \bottomrule
		\end{tabular}%
	}
	\caption{\textit{Number of samples taken from each model by the MOSAP-optimal MLBLUE, MLMC and MFMC methods for Problem~2. ``H-H'' is short for ``Hodgkin-Huxley'' and ``F-N'' for ``FitzHugh-Nagumo'', while ``coarse'', ``med.'', and ``fine'' correspond to the grid and timestep refinement levels $i=1,2,3$ respectively. MLMC and MFMC only use the Hodgkin-Huxley PDE models, while the fine- and medium-grid FitzHugh-Nagumo PDE models are discarded by all methods (hence not shown).}}
	\label{tab:HH_number_of_samples}
\end{table}

\begin{figure}[h!]
	\centering
	\includegraphics[width=0.8\textwidth]{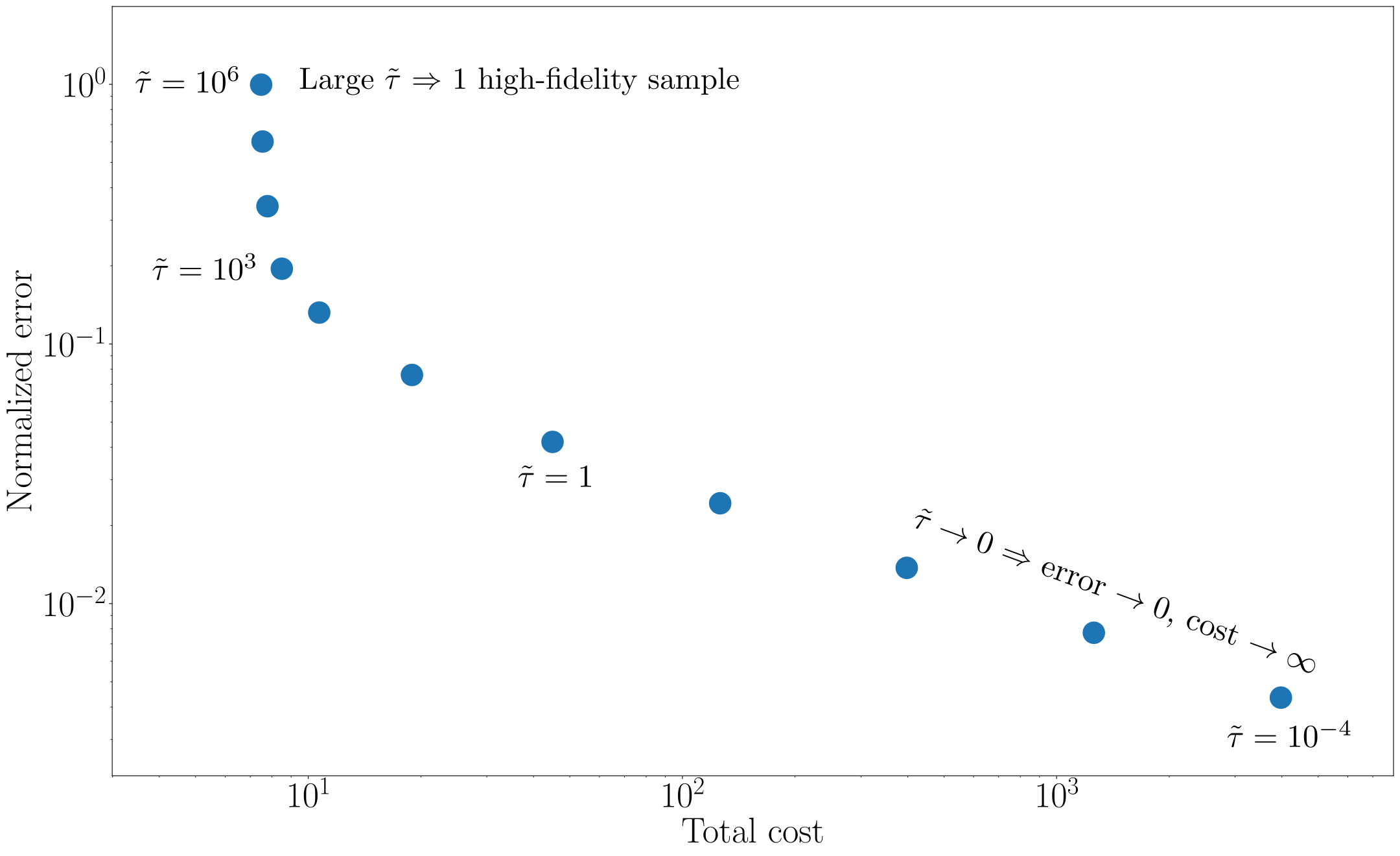}
	\caption{\textit{Pareto-front optimization of MLBLUE for Problem 2. Here $\tilde{\tau}=\tau\lVert \bm{c} \rVert_2$, where $\tau$ is the Pareto parameter from \eqref{eq:MLBLUE_MOMOSAP_Pareto} and $\bm{c}$ is the model cost vector. From right to left, the points in the figure correspond to $\tilde{\tau}=10^i$ for $i=-4,\dots,6$.}}
	\label{fig:pareto_front_HH}
	\vspace{-12pt}
\end{figure}

\paragraph{Estimator setup.} We use the same setup as for Problem 1. Again, we set the statistical error tolerances for the SDP MOSAP to be $\varepsilon_s=10^{-3}\sqrt{\V[p^s_1]}$ for $s=1,\dots m=4$ where $\V[p^s_1]$ is estimated from pilot samples.

\paragraph{Results: total estimation cost and number of samples.} Results are shown in Figure \ref{fig:HH_cost_complexity}, where we show the total cost of each method, and in Table \ref{tab:HH_number_of_samples}, where we report the number of samples taken for each model. For Problem 2 MLBLUE is roughly five times faster than MLMC and two times faster than MFMC. Again, MLBLUE employs a different model selection strategy than the other methods. In fact, MLBLUE is the only method that uses the FitzHugh-Nagumo models and the ODE models, which are much cheaper. These models yield poor approximations to some of the QoIs, causing the model selection procedure (cf.~Subsection \ref{subsubsec:NS}) of MLMC and MFMC to discard them. In contrast, multi-output MLBLUE includes these models in the estimation and exploits their correlations, thus considerably reducing the total number of high-fidelity samples.

\begin{remark}
	While not required for the examples considered in this section, the model selection procedure must account for models that do not admit some of the QoIs as output. For multi-output MLBLUE, this situation is handled automatically by constructing the $\Psi_s$ matrices as described at the beginning of Section \ref{sec:multi-output}. Nevertheless, manual model pruning could still be beneficial to restrict the number of possible model groupings, cf.~Remark \ref{rem:max_group_size}. For MLMC and MFMC, a simple (though suboptimal) option is to simply discard such models from the estimator. Another option would be to combine multiple single-output MLMC/MFMC estimators within a single joint optimization procedure. We leave the investigation of such a strategy to future work.
\end{remark}

\paragraph{Results: Pareto frontier}
We employ the MOSAP \eqref{eq:MLBLUE_MOMOSAP_Pareto} to trace the Pareto frontier of MLBLUE for Problem 2. Results are shown in Figure \ref{fig:pareto_front_HH} where we plot the Pareto points  corresponding to $\tau=\tilde{\tau}\lVert \bm{c} \rVert_2^{-1}$, where $\tilde{\tau}=10^i$ for $i=-4,\dots,6$. Other than the chosen values of $\tilde{\tau}$, the quantities in the figure are computed in the same way as for Figure \ref{fig:pareto_front_NS}. The behavior is similar to what we observed for Problem 1 with the exception that the transition to the asymptotic regime appears to be smoother for this problem.

\subsubsection{Numerical optimization of the MOSAPs}
The objective of this subsection is to benchmark the efficiency and robustness of the new SDP MOSAP formulations against the standard formulations. We consider Problems 1 and 2, and all possible model groupings of size at most $\kappa$ with $\kappa\in\{3,5,7\}$ (cf.~Remark \ref{rem:max_group_size}).

We solve the SDPs \eqref{eq:MLBLUE_MOSAP_SDP}, \eqref{eq:MLBLUE_MOSAP_SDP_2}, \eqref{eq:MLBLUE_MOMOSAP} and \eqref{eq:MLBLUE_MOMOSAP_2} using the \texttt{conelp} solver in the \texttt{CVXOPT} open-source conic optimization software package\footnote{Available at \url{https://cvxopt.org/}.} \cite{cvxopt}. For the standard MOSAP formulations \eqref{eq:MLBLUE_MOSAP_1}, \eqref{eq:MLBLUE_MOSAP_2}, \eqref{eq:MLBLUE_MOMOSAP_std}, and \eqref{eq:MLBLUE_MOMOSAP_2_std}, we use the \texttt{trust-constr} trust-region constrained optimization solver from \texttt{SciPy}\footnote{Available at \url{https://scipy.org/}.} and the open-source black-box interior-point optimization solver  \texttt{IPOPT}\footnote{Available at \url{https://coin-or.github.io/Ipopt/}.}  \cite{wachter2006implementation}. We set an absolute error tolerance of $10^{-6}$ for all optimization solvers. For the budget-constrained MOSAPs we set the budget $b=10^4\max(\bm{c})$, where $\bm{c}$ is the cost vector for all problems. For the statistical-error-constrained MOSAPs we select $\varepsilon_s=10^{-3}\sqrt{\V[p^s_1]}$, for $s=1,\dots,m$, where $m$ is the number of QoIs in the problem and $\sqrt{\V[p^s_1]}$ is the estimated standard deviation of the $s$-th QoI of the high-fidelity model.

We report that the \texttt{trust-constr} method fails to converge to the optimal solution in $5000$ iterations for both problems and all values of $\kappa$. \texttt{IPOPT} does converge, but often fails unless its solver parameters are set accordingly. In particular, we had to set \texttt{bound\_relax\_factor} to $10^{-30}$ and \texttt{honor\_original\_bounds} to \texttt{yes}. The \texttt{CVXOPT} SDP solver always converges in less than $100$ iterations. The only problem for which \texttt{CVXOPT} requires a change in the default solver parameters is the budget-constrained MOSAP formulation of Problem 1 for which setting the feasibility tolerance \texttt{feastol} to $10^{-3}$ is needed to ensure convergence.

We show the \texttt{IPOPT} and \texttt{CVXOPT} SDP solver timings in Table \ref{tab:optimization_results}. The new SDP formulations can robustly be solved in a matter of seconds or minutes depending on the problem size and are much faster to solve than the standard formulations.

\begin{table}[h!]
	\centering
	\begin{tabular}{@{}cccccccc@{}}
		\toprule
		\multirow{2}{*}{Problem}                                                         & \multirow{2}{*}{MOSAP}   & \multicolumn{3}{c}{Minimize statistical error}      & \multicolumn{3}{c}{Minimize total cost} \\
		&                          & $\kappa=3$      & $\kappa=5$      & $\kappa=7$                          & $\kappa=3$      & $\kappa=5$      & $\kappa=7$         \\ \midrule
		\multicolumn{1}{l|}{\multirow{2}{*}{Problem $1$ ($12$ models, $6$ QoIs)}}                                   & \multicolumn{1}{c|}{SDP} & $1.67$   & $66$    & \multicolumn{1}{c|}{$540$}   & $1.2$      & $37$       & $340$      \\
		\multicolumn{1}{c|}{}%($b=10^4\max(\bm{c})$, $\varepsilon_{s}=10^{-3}\sqrt{\V[p^{s}_1]}$)}
		& \multicolumn{1}{c|}{std} & $1.96$        & $215$        & \multicolumn{1}{c|}{$2500$}        & $3.04$           & $627$          & $5000$           \\ \midrule
		\multicolumn{1}{l|}{\multirow{2}{*}{Problem $2$ ($12$ models, $4$ QoIs)}}                                   & \multicolumn{1}{c|}{SDP} & $1.42$   & $56$    & \multicolumn{1}{c|}{$518$}   & $1.14$      & $44$       & $292$      \\
		\multicolumn{1}{c|}{}%($b=10^4\max(\bm{c})$, $\varepsilon_{s}=10^{-3}\sqrt{\V[p^{s}_1]}$)}
		& \multicolumn{1}{c|}{std} & $1.54$        & $163$        & \multicolumn{1}{c|}{$2580$}        & $1.85$           & $314$           & $8983$           \\ \bottomrule
	\end{tabular}
	\caption{\textit{CPU time (in seconds) needed to solve the MLBLUE MOSAPs to optimality on a Dell XPS 13 laptop (processor: Intel \textnormal{i7-1185G7} $3$\textnormal{GHz}, RAM: $32$\textnormal{GB}). The SDP MOSAPs are solved with \texttt{CVXOPT}, while the standard MOSAP formulations are solved with \texttt{IPOPT}.}}
	\label{tab:optimization_results}
\end{table}

\begin{remark}
	\label{rem:integer_proj}
	MLBLUE is more sensitive than MLMC and MFMC to integer projections. In fact, we note that taking the ceiling of the real-valued MOSAP solution $\bm{n}^*$ is generally a bad idea since $\bm{n}^*$ has often many entries that are not close to an integer. A better approach, which we employ in \texttt{BLUEST}, is to take the floor or ceiling of all such entries and try out all feasible combinations. This procedure can be accelerated and made reliable via heuristics that we do not describe here for the sake of brevity. We leave the investigation of a more sophisticated approach, using perhaps a branch and bound algorithm, to future research.
\end{remark}

\subsection{Coupling restrictions}
In this section we consider the third and last test problem. The objective here is to demonstrate that MLBLUE can automatically handle coupling restrictions. In particular, Problem 3 represents a common scenario in scientific computing in which the high-fidelity models are prohibitively expensive and cannot be simulated more than a handful of times. We will also investigate two different approaches for model covariance estimation in this context.

\subsubsection{Problem 3: linear random diffusion with restricted high-fidelity samples}
\label{subsec:problem1}

\paragraph{Problem description.} As our final test problem we choose a standard linear diffusion partial differential equation (PDE) with random field diffusion coefficient in 2D:
\begin{align}
\label{eq:problem1}
\begin{dcases}
\begin{array}{lll}
- \nabla\cdot(a(\bm{x},\omega)\nabla u(\bm{x},\omega)) = f(\bm{x}), & \bm{x}\in \mathcal{D}_1=(0,1)^2,\\
u(\bm{x},\omega)=g(\bm{x},\omega), & \bm{x}\in\partial \mathcal{D}_1.%, & \omega\in\Omega.
\end{array}
\end{dcases}
\end{align}
Here, $f(\bm{x})=\exp(-\lVert\bm{x}-0.5\rVert_2^2)$ and $g(\bm{x},\omega)=\theta_1(\omega)x_1^2 + \theta_2(\omega)x_2^2 + 2\theta_3(\omega)x_1x_2$, where $\theta_1,\theta_2,\theta_3$ are independent zero-mean Gaussian random variable with variance $2^{-4}$. We set the diffusion coefficient to be $a=e^z$, where $z=z(\bm{x},\omega)$ is a Gaussian field with a Mat\'ern covariance matrix with smoothness parameter $\nu=1$ and correlation length $\lambda=0.2$. We choose the mean and standard deviation of $z$ so that $\E[a]=1.0$, and $\V[a]^{1/2}=0.3$ for all $\bm{x}\in\mathcal{D}_1$. For the definition and properties of Mat\'ern Gaussian fields, we refer the reader to \cite{PetterAbrahamsen1997}. For sampling the Mat\'ern field, we use the stochastic PDE approach by Lindgren et al.~\cite{Lindgren2011} combined with the linear cost-complexity sampling technique by Croci et al.~\cite{CrociWhiteNoise2018,croci2021multilevel}.

\paragraph{Quantity of interest.} We consider a single QoI given by $p=\int_{\mathcal{D}_1} u^2 \text{ d}x$.

\paragraph{Model set.} We define our model set by simple grid refinement using a hierarchy of $\ell=7$ uniform triangular grids so that model $i\in\{1,\dots,7\}$ corresponds to a grid with $2^{8-i}$ cells in each spatial direction.

\paragraph{Solver and model costs.} We solve \eqref{eq:problem1} with the finite element (FE) method using the open-source FE library \texttt{FEniCS} \cite{LoggEtAl2012} and piecewise-linear Lagrange elements using the conjugate gradient method preconditioned with algebraic multigrid. Cost estimation via CPU timings yields a roughly linear solver cost complexity and we thus define the model sampling costs to be the number of degrees of freedom in each FE model for simplicity.

\paragraph{Coupling restriction.}
We restrict the number of samples of the two highest-fidelity models to a fixed number $n_{\text{HF}}\in\{2,4,8,16\}$ and investigate how the results vary with $n_{\text{HF}}$.
For a study of how MLBLUE performs for a similar diffusion problem, but without any sampling restrictions, we refer the reader to the original papers by Schaden and Ullmann \cite{schaden2020multilevel,schaden2021asymptotic}. We remark that their work also includes efficiency comparisons between MLBLUE and other estimators, including MLMC, MFMC and ACV methods, which MLBLUE outperforms.

\paragraph{Covariance estimation.} We compare two approaches for estimating covariances.
\begin{itemize}
\item[a)] We take $n_{\text{HF}}$ pilot samples of all models and compute sample covariances with these samples.
  \item[b)] We take $16$ pilot samples of the low-fidelity models, compute their sample covariances, then use extrapolation to approximate the covariances of the high-fidelity models. Specifically:
\begin{enumerate}
	\item With the low-fidelity pilot samples, compute all low-fidelity model covariances, as well as the variance of the difference between model pairs, i.e., $\V[p_i-p_j]$ for all $i,j\in\{3,\dots,7\}$, $i\neq j$.
	\item Draw a single deterministic sample of all models by fixing $a=1$, $\theta_1=\theta_2=\theta_3=0$ (their average values), and estimate the convergence rate $r$ with respect to the model degrees of freedom.
	\item Use Richardson extrapolation with rate $2r$ to extrapolate the values of $\V[p_i]$ and $\V[p_i - p_{i+j}]$ for $i=1,2$ and $j=1,\dots,\bar{d}$ from their low-fidelity values. Here $\bar{d}$ is an integer parameter satisfying $1\leq \bar{d}\leq \ell-2$ that determines how many covariances are being extrapolated: we estimate the covariances between each high-fidelity model and the next $\bar{d}$ lower-fidelity models. We will be varying $\bar{d}$ in our experiments.
	\item Approximate the high-fidelity model covariances using the formula $C(p_i,p_{i+j}) = \frac{1}{2}(\V[p_i]+\V[p_{i+j}]-\V[p_i-p_{i+j}])$ for $i=1,2$ and $j=1,\dots,\bar{d}$.
\end{enumerate}
In step 2 we use the fact that for PDEs with random coefficients the variance convergence rate is typically twice the deterministic convergence rate, see e.g., \cite{Cliffe2011}. Note that in steps 3-4 we do not necessarily approximate the covariances between all models, hence some model groupings will not be allowed. Nevertheless, restricted model couplings are not problematic for MLBLUEs.
\end{itemize}

\paragraph{Estimator setup.} We proceed as follows. First, we fix a value of $n_{\text{HF}}\in\{2,4,8,16\}$, and either Approach a) or b). Then, we solve the budget-constrained MOSAP \eqref{eq:MLBLUE_MOSAP_SDP} with budget given by the cost of sampling all models $2^7=128$ times, with maximum allowed model group size $\kappa=4$ (cf.~Remark \ref{rem:max_group_size}), and with the additional restriction that the high-fidelity models cannot be sampled more than $n_{\text{HF}}$ times.

\paragraph{Results: estimator efficiency.} After denoting with $\theta$ the set of pilot samples taken, we compute the log of the resulting MLBLUE normalized estimator efficiency:
\begin{align}
\label{eq:normalized_efficiency}
\log_{10}\left(\frac{E(\theta)}{E(\theta_{\text{ex}})}\right)=\log_{10}\left(\frac{{\varepsilon}^2(\theta_{\text{ex}})}{\V[\hat{\mu}_1]}\right),
\end{align}
where $E(\theta)$ is the estimator efficiency defined as
\begin{align}
\label{eq:efficiency}
E(\theta) = \frac{1}{b\V[\hat{\mu}_1]},
\end{align}
and $\hat{\mu}_1=\hat{\mu}_1(\theta)$ is the MLBLUE constructed from a single realization $\theta$ of the pilot samples, with $\V[\hat{\mu}_1](\theta)$ being its variance.  Here $E(\theta_{\text{ex}})=(b\varepsilon^2(\theta_{\text{ex}}))^{-1}$ and ${\varepsilon}^2(\theta_{\text{ex}})$ is the variance of a MLBLUE estimator with $n_{\text{HF}}=\infty$ constructed from an exact model covariance estimation using $1000$ pilot samples. $E(\theta_{\text{ex}})$ essentially represents the efficiency of the best MLBLUE that can be constructed for this problem with the prescribed budget. We call the quantity in \eqref{eq:efficiency} ``estimator efficiency'' since the inverse product between a Monte Carlo estimator cost and its variance indicates whether an increased budget leads to a more accurate estimator. In fact, the accuracy of any Monte Carlo estimator for an unbiased QoI is proportional to the inverse square root of the cost, i.e., $\varepsilon = O(\text{cost}^{-1/2})$. Therefore, the efficiency $E = (\text{cost }{\varepsilon}^2)^{-1}$ is always an order-$1$ quantity in the estimator cost. Nevertheless, a more efficient method will lead to a higher accuracy, a smaller value of $\varepsilon$, and a higher value of $E$. Note that the closer \eqref{eq:normalized_efficiency} is to $0$, the closer the efficiency of the constructed MLBLUE will be to the best estimator efficiency.

\begin{figure}[h!]
	\centering
	\includegraphics[width=\textwidth]{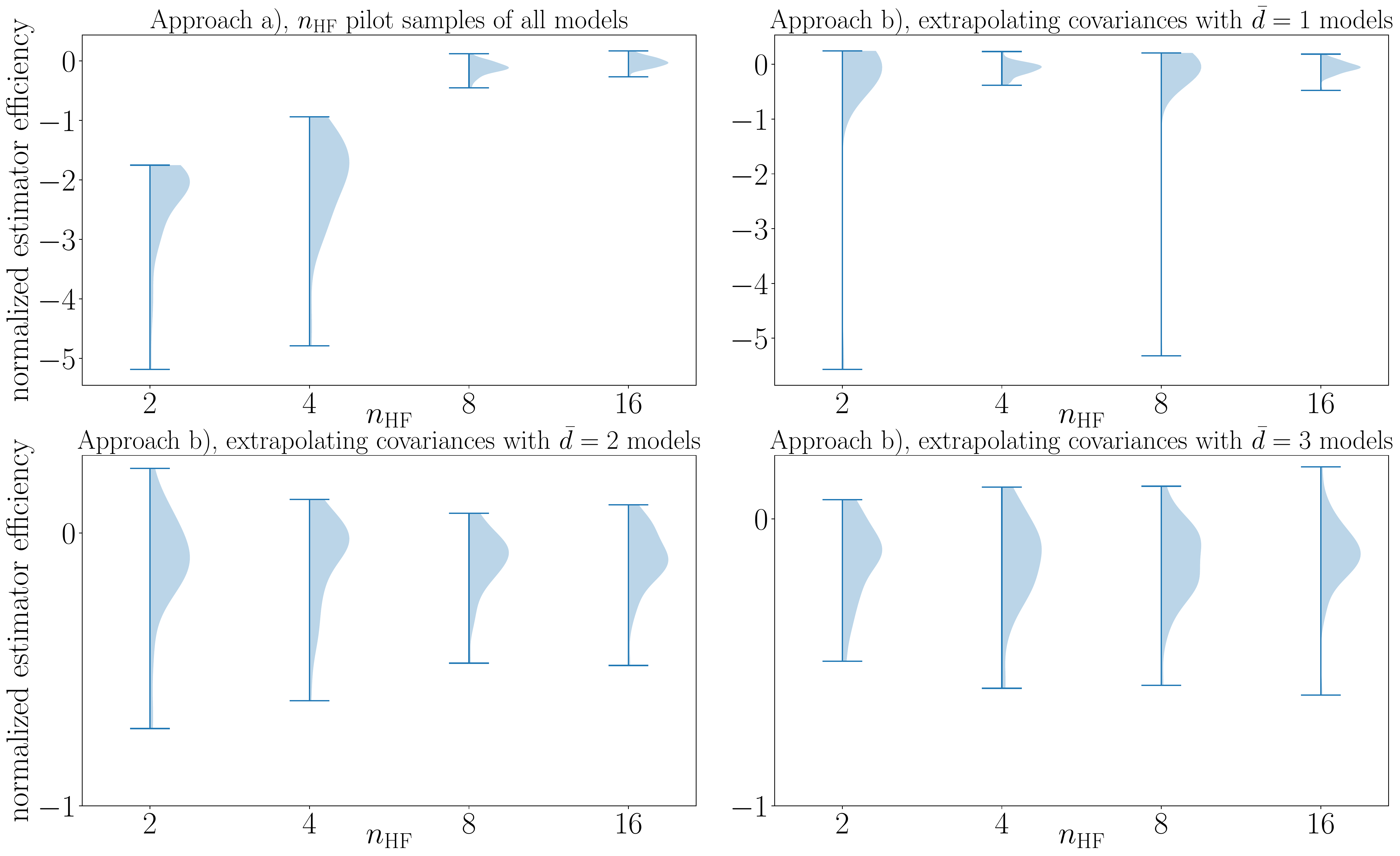}
	\caption{\textit{Half-violin plots of the normalized estimator efficiency \eqref{eq:normalized_efficiency} for the approaches for Problem~3 described in Section \ref{subsec:problem1} versus the number of restricted high-fidelity samples $n_{\text{HF}}$. The light-blue areas depict kernel-density estimates of the probability density function of \eqref{eq:normalized_efficiency} with respect to $\theta$. Values close to $0$ indicate near-best efficiency.}}
	\label{fig:restrictions_matern}
	\vspace{0pt}
\end{figure}

Since poor values of $\theta$ may lead to underestimating or overestimating the estimator variance, we draw $50$ independent realizations of the estimator to approximate $\V[\hat{\mu}_1](\theta)$. In Figure \ref{fig:restrictions_matern}, we study the effect of the high-fidelity sample restrictions and of the model covariance estimation approaches described in Section \ref{subsec:problem1} across $50$ independent realizations of the pilot samples $\theta$. We see how Approach a) leads to poor estimator performance for $n_{\text{HF}}\in\{2,4\}$, suggesting that this is not a good approach when the high-fidelity samples are severely restricted. Approach b) with $\bar{d}=1$ performs well in general, but it is not robust to an inaccurate pilot sample estimation. In fact, some realizations of $\theta$ lead to a poor estimator efficiency. On the other hand, Approach~b) for $\bar{d}>1$ yields good results for all values of $n_{\text{HF}}$ indicating that extrapolation of multiple model cross-covariances, when possible, should be preferred. We report that for both approaches the estimator efficiency does not improve further as $n_{\text{HF}}$ increases beyond $16$. This result indicates that MLBLUE can handle coupling restrictions with little effect on its efficiency.

\begin{remark}
	The values of \eqref{eq:normalized_efficiency} in Figure \ref{fig:restrictions_matern} that are larger than $0$ are likely due to noise in the estimation of $\V[\hat{\mu}_1]$. To produce this experiment we sampled $50^2=2500$ independent estimators. Making the estimation of the estimator variance more accurate, while prohibitively expensive, would hardly change the qualitative behavior of these results.
\end{remark}

\begin{remark}
	A consequence of not knowing all correlations in Approach b) with $\bar{d}=1$ is that the only allowed model groups that include the two highest-fidelity models $\mathcal{M}_1$ and $\mathcal{M}_2$ are
	\begin{align*}
	\{\mathcal{M}_1\},\ \{\mathcal{M}_2\},\ \{\mathcal{M}_1, \mathcal{M}_2\},\ \{\mathcal{M}_2, \mathcal{M}_3\},
	\end{align*}
	i.e., the two models can only be sampled on their own or coupled in pairs.
	On the other hand, all the groupings between the low-fidelity models $\mathcal{M}_3,\dots, \mathcal{M}_7$ are allowed. This coupling structure could be seen as a way of combining a MLMC estimator on the high-fidelity models (coupled in pairs) with a MFMC estimator on the low-fidelity models (all coupled together), cf.~Example \ref{ex:MLBLUE_flexible}. The MLBLUE algorithm can automatically find such an estimator with little input from the user.
\end{remark}

\section{Conclusions}
\label{sec:conclusions}

MLBLUEs are optimal across all multilevel and multifidelity linear unbiased estimators and they should thus be preferred in principle over any other such method. However, MLBLUEs (as well as all multilevel/multifidelity Monte Carlo methods) may lose efficiency if their MOSAP is not solved exactly or if the model covariance estimation is prohibitive.

In this paper we addressed both issues. We formulated the MLBLUE MOSAPs as semi-definite programs that can be solved reliably and efficiently and we showed that these estimators are robust with respect to limited and/or restricted pilot samples. Further, we have extended MLBLUE to the multi-output case by concurrently setting up multiple estimators for all QoIs, thus making MLBLUEs the only methods with an optimal strategy for handling multiple outputs. Consequently, this work enables the fast, reliable design and setup of MLBLUEs for forward UQ problems involving one or more QoIs. Numerical experimentation shows the resulting MLBLUEs to be more efficient than MLMC and MFMC. MLBLUEs can also deal efficiently with heterogeneous model sets.

Additional work is still needed to design a comprehensive MLBLUE algorithmic framework: Integer programming strategies for the MOSAP could be beneficial since the continuous relaxation of the MOSAP may be far from optimal for some problems (cf.~Remark \ref{rem:integer_proj}). Furthermore, specific strategies should be designed to decrease the cost of offline estimation and/or allow pilot sample reuse. For instance, QoIs with large kurtosis require a large number of pilot samples \cite{giles2015multilevel} which, if re-used within MLBLUE, would make the estimator biased \cite{schaden2020multilevel}, and hence not necessarily optimal. Finally, MLBLUE advanced estimation strategies for rare-event simulation, nested expectations, and non-smooth QoIs are yet to be investigated.

While the most efficient method for a given application is, in general, problem-dependent, single- and multi-output MLBLUEs are optimal whenever the above challenges are not relevant and should thus be the methods of choice. MLBLUE is currently the only method with an optimal multi-output strategy and that is capable of automatically addressing model selection, making it a powerful addition to the forward UQ toolbox.

\printbibliography

\end{document}